\documentclass{siamltex}
\usepackage[dvips]{graphicx}
\usepackage{amsfonts}
\usepackage{amssymb}
\usepackage{color}
\usepackage{subfigure}
\usepackage{amsmath,latexsym}

\newtheorem{Assumption}{Assumption}[section]
\newtheorem{example}{Example}

\newcommand{\lbl}[1]{\label{#1}}

\newtheorem{remark}{Remark}[section]
\def\eps{\varepsilon}
\def\pa{\partial}
\def\ds{\mathrm{d}\mathbf{s}}

\def\nat{\nabla_{\Gamma}}
\def\nath{\nabla_{\Gamma_h}}
\def\wt{\mathbf{w}_{\Gamma}}

\def\divth{\mathrm{div}_{\Gamma_h}}
\def\Div{\operatorname{div}}
\def\Fh{\mathcal{F}_h}
\def\Eh{\mathcal{E}_h}

\def\bfx{\mathbf{x}}
\def\bx{\mathbf{x}}
\def\be{\mathbf{e}}

\def\bw{\mathbf{w}}
\def\bn{\mathbf{n}}
\def\bfn{\mathbf{n}}
\def\bfp{\mathbf{p}}
\def\bfP{\mathbf{P}}
\def\bbR{\mathbb{R}}

\def\ipue{(I_h u^e)|_{\Gamma_h}}
\def\mbf{\mathbf}

\title{A stabilized finite element method for advection-diffusion
equations on surfaces\footnotemark[1]}%

\author{Maxim A. Olshanskii\thanks{Department of Mathematics, University of Houston, Houston, Texas 77204-3008 and Department of Mechanics and Mathematics, Moscow State
University, Moscow, Russia, 119899
(Maxim.Olshanskii@mtu-net.ru).}
\and Arnold Reusken\thanks{Institut f\"ur Geometrie und Praktische  Mathematik, RWTH-Aachen
University, D-52056 Aachen, Germany (reusken@igpm.rwth-aachen.de,xu@igpm.rwth-aachen.de).}
\and Xianmin Xu$^\dag$\thanks{LSEC, Institute of Computational
  Mathematics and Scientific/Engineering Computing,
  NCMIS, AMSS, Chinese Academy of Sciences, Beijing 100190, China}
}

\begin{document}
\maketitle

\begin{abstract}
A recently developed Eulerian finite element method is applied to solve
advection-diffusion equations posed on hypersurfaces. When transport processes on a surface
dominate over diffusion, finite element methods tend to be unstable unless the mesh
is sufficiently fine. The paper introduces a stabilized finite element formulation based on the SUPG technique.
An error analysis of the method is given.
Results of numerical experiments are presented that illustrate the performance of the stabilized method.
\end{abstract}



\pagestyle{myheadings}
\thispagestyle{plain}
\markboth{M. A. OLSHANSKII, A.~REUSKEN, AND X.~XU  }{A FEM FOR ADVECTION-DIFFUSION EQUATIONS ON SURFACES}

\section{Introduction}

Mathematical models involving  partial differential
equations posed on hypersurfaces occur in many applications.
Often surface equations are coupled  with other equations that are
formulated  in a (fixed) domain which contains the surface.
This happens, for example, in common models of
multiphase fluids dynamics  if one  takes so-called surface active agents
into account \cite{GrossReuskenBook}. The surface transport of such surfactants is typically driven by convection and surface diffusion and the relative strength of these two is measured by the dimensionless surface Peclet number $Pe_s=\frac{UL}{D_s}$. Here $U$ and $L$ denote typical velocity and lenght scales, respectively, and $D_s$ is the surface diffusion coefficient. Typical surfactants have surface diffusion coefficients in the range $D_s \sim 10^{-3}-10^{-5}~ cm^2/s$ \cite{Agrawal1988}, leading to (very) large surface Peclet numbers in many applications. Hence, such applications result in advection-diffusion equations on the surface
with dominating advection terms. The surface may evolve in time and
be available only implicitly (for example, as a zero level of a  level set function).

It is well known that  finite element  discretization methods for advection-diffusion problems  need
an additional  stabilization mechanism, unless the mesh size is sufficiently small to resolve
boundary and internal layers in  the solution of the differential equation.  For the planar case, this topic has been extensively studied
in the literature and a variety of stabilization methods has been developed, see, e.g., \cite{TobiskaBook}.
We are, however, not aware of any studies of stable finite element methods for
 advection-diffusion equations posed on surfaces.

In the past decade the study of numerical methods for  PDEs on  surfaces has been a rapidly growing research area.
The development of  finite element methods
for solving elliptic equations on surfaces can be traced back to the paper \cite{Dziuk88},
which considers a piecewise polygonal surface and uses a
finite element space on a triangulation of this discrete surface. This approach has been further analyzed and extended in several directions, see, e.g.,~\cite{Dziuk011} and the references therein.
Another approach has  been introduced in \cite{Deckelnick07} and builds on the ideas of \cite{Bertalmio01}. The method in that paper applies to cases in which the surface is
given implicitly by some level set function and the key idea is to solve the partial differential
equation on a narrow band around the surface. Unfitted finite element spaces on this narrow band are
used for discretization. Another surface finite element method  based on an outer (bulk) mesh has been introduced in \cite{Reusken08} and further studied in \cite{OlshanskiiReusken08,DemlowOlshanskii12}.
The main idea of this method is to use  finite element spaces that are induced by
triangulations of an outer domain  to discretize the partial differential
equation on the surface by considering \textit{traces} of the bulk finite element space on the surface, instead of
extending the PDE off the surface, as in \cite{Bertalmio01,Deckelnick07}.
The method is particularly suitable for problems in which
the surface is given implicitly by a level set or VOF function and in which there is a
coupling with a differential equation in a fixed outer domain.
If in such problems one uses finite element techniques for the discetization of
equations in the outer domain, this setting immediately results in an easy to implement discretization method for the surface equation.  The approach does not require additional surface
elements.

In this paper we reconsider the volume mesh finite element method from \cite{Reusken08} and study a new aspect, that has not been studied in the literature so far, namely the stabilization of advection-dominated problems.
We restrict ourselves to the case of a stationary surface.  To stabilize the discrete problem for the case
of large mesh Peclet numbers,  we introduce a surface variant of the SUPG method.
For  a class of stationary advection-diffusion equation  an error analysis is presented. Although the convergence of the method is studied using a SUPG norm similar to the planar case~\cite{TobiskaBook}, the analysis is not standard and contains new ingredients:
Some new approximation properties for the traces of finite elements are needed and geometric errors
require special control. The main theoretical result is given in Theorem~\ref{Th1}. It yields an error estimate
in the SUPG norm which  is almost robust in the sense that   the dependence on the Peclet number is  mild. This dependence is due to some  insufficiently controlled geometric errors, as will be explained in section~\ref{sec_disc}.

The remainder of the paper is organized as follows. In
section~\ref{S_cd}, we recall equations for  transport-diffusion processes on surfaces and present
the stabilized finite element method.
Section~\ref{sectanalysis} contains the  theoretical results of
the paper concerning the approximation properties of the finite element space
and discretization error bounds for the finite element method.  Finally, in section~\ref{sectexperiments}
results of numerical experiments are given for both  stationary and time-dependent advection-dominated surface transport-diffusion equations, which show that the stabilization performs well and that numerical results are consistent with what is expected from the  SUPG method in the planar case.

\section{Advection-diffusion equations on surfaces} \label{S_cd}
Let $\Omega$ be an open domain in $\mathbb{R}^3$ and $\Gamma$ be a connected $C^2$ compact
hyper-surface contained in $\Omega$.
For a sufficiently smooth function $g:\Omega\rightarrow \mathbb{R}$ the tangential derivative
at $\Gamma$ is defined by
\begin{equation}
 \nabla_{\Gamma} g=\nabla g-(\nabla g\cdot \mathbf{n}_{\Gamma})\mathbf{n}_{\Gamma},\lbl{e:2.1}
\end{equation}
where $\mathbf{n}_{\Gamma}$ denotes the unit normal to $\Gamma$. Denote by $\Delta_{\Gamma}$
the { Laplace-Beltrami operator} on $\Gamma$.
Let $\mathbf{w}:\Omega\rightarrow\mathbb{R}^3$ be a given divergence-free ($\Div\bw=0$) velocity field in $\Omega$. If the surface
$\Gamma$ evolves with a normal velocity of $\bw\cdot\bn_{\Gamma}$, then the conservation of a scalar quantity
$u$ with a diffusive flux on $\Gamma(t)$ leads to the surface PDE:
\begin{equation}
\dot{u} + (\Div_\Gamma\bw)u -\varepsilon\Delta_{\Gamma} u=0\qquad\text{on}~~\Gamma(t),
\label{transport}
\end{equation}
where $\dot{u}$ denotes the advective material derivative, $\eps$ is the diffusion coefficient. In \cite{Dziuk07} the problem \eqref{transport} was shown to be well-posed in a suitable weak sense.

In this paper, we study a finite element method for an advection-dominated problem on a \textit{steady} surface.     Therefore, we assume $\bw\cdot\bn_{\Gamma}=0$, i.e. the advection velocity is everywhere tangential to the surface. This and  $\Div\bw=0$ implies $\Div_\Gamma\bw=0$, and the surface advection-diffusion equation takes the form:
\begin{equation} u_t+ \mathbf{w}\cdot\nabla_{\Gamma} u -\varepsilon\Delta_{\Gamma} u=0
\qquad\text{on}~~\Gamma.
\lbl{e:2.2}
\end{equation}
Although the methodology and numerical examples of the paper are applied to the equations \eqref{e:2.2},
the error analysis will be presented for the stationary problem
\begin{equation}
  -\varepsilon\Delta_{\Gamma} u+\mathbf{w}\cdot\nabla_{\Gamma} u + c(\bfx)u =f\qquad\text{on}~~\Gamma, \lbl{e:2.3}
\end{equation}
with  $f\in L^2(\Gamma)$ and $c(\bfx)\ge0$. To simplify the presentation we assume $c(\cdot)$ to be constant, i.e. $c(x)=c \geq 0$. 
The analysis, however, also applies to non-constant $c$, cf. section~\ref{sec_disc}.
 Note that  \eqref{e:2.2} and \eqref{e:2.3} can be written in intrinsic surface quantities, since
$\mathbf{w}\cdot\nabla_{\Gamma} u = \mathbf{w}_{\Gamma}\cdot\nabla_{\Gamma} u$, with the tangential velocity $\mathbf{w}_{\Gamma}=\mathbf{w}-(\mathbf{w}\cdot\mathbf{n}_{\Gamma})\bfn_{\Gamma}$.
We assume  $\wt\in H^{1,\infty}(\Gamma)\cap L^{\infty}(\Gamma)$ and scale the equation so that
$\|\wt\|_{L^\infty(\Gamma)}=1$ holds. Furthermore, since we are interested in the advection-dominated case we take $\eps\in(0,1]$. Introduce the
bilinear form and the functional:
\begin{equation*}
\begin{aligned}
a(u,v)&:=\eps\int_{\Gamma}\nabla_{\Gamma} u\cdot\nabla_{\Gamma} v \, \ds+\int_{\Gamma}(\bw\cdot\nat u) v \, \ds
+\int_{\Gamma}c \, uv\, \ds, \\
f(v)&:=\int_{\Gamma}fv\, \ds.
\end{aligned}
\end{equation*}
The weak formulation of  \eqref{e:2.3} is as follows: Find $u\in V$ such that
\begin{equation}
a(u,v)=f(v) \qquad \forall v\in V,\lbl{e:2.5}
\end{equation}
with
$$V= \left\{
\begin{array}{ll}
 \{v\in H^1(\Gamma)\ |\ \int_{\Gamma} v\, \ds=0\} & \hbox{if } c= 0,\\
H^1(\Gamma) & \hbox{if } c>0.
     \end{array}
     \right.
$$
Due to  the Lax-Milgram lemma,  there exists a unique solution of \eqref{e:2.5}.
For the case $c=0$  the following Friedrich's  inequality~\cite{Sobolev} holds:
\begin{equation}
\|v\|_{L^2(\Gamma)}^2 \le C_F \|\nabla_\Gamma v\|_{L^2(\Gamma)}^2\quad
\forall~v\in V.
\lbl{FdrA}
\end{equation}

\subsection{The stabilized volume mesh FEM} \label{sectSUPG}
 In this section, we  recall the volume mesh FEM introduced in \cite{Reusken08} and describe its SUPG type stabilization.

Let $\{\mathcal{T}_h\}_{h>0}$ be a family of tetrahedral triangulations of
the domain $\Omega$. These triangulations
are assumed to be regular, consistent and stable. To simplify the presentation we assume that this family of triangulations is {quasi-uniform}. The latter assumption, however, is not essential for our analysis.
We assume that for each $\mathcal{T}_h$  a polygonal approximation  of
$\Gamma$, denoted by $\Gamma_h$, is given:
$\Gamma_h$ is a $C^{0,1}$ surface without  boundary and $\Gamma_h$ can
be partitioned in planar triangular segments.
It is important to note  that $\Gamma_h$ is not a ``triangulation of $\Gamma$'' in the  usual sense
(an $O(h^2)$ approximation of $\Gamma$, consisting of regular triangles). Instead,
we (only) assume that $\Gamma_h$ is \emph{consistent with the outer triangulation} $\mathcal{T}_h$
in the following sense.  For any
tetrahedron $S_T\in\mathcal{T}_h$ such that $\mathrm{meas}_2(S_T\cap\Gamma_h)>0$, define
$T=S_T\cap\Gamma_h$. We assume  that every $T\in\Gamma_h$  is a \textit{planar} segment
and thus it is  either a triangle or a quadrilateral. Each quadrilateral segment can be divided into two triangles, so we may assume that every $T$ is a triangle. An illustration of such a triangulation is given in Figure~\ref{fig:tri}. The results shown in this figure are obtained by representing a sphere $\Gamma$   implicitly by its signed distance function, constructing the piecewise linear nodal interpolation of this distance function on a uniform tetrahedral triangulation $\mathcal{T}_h$ of $\Omega$ and then considering the zero level of this interpolant.
\begin{figure}[ht!]
\begin{center}
\centering
  \includegraphics[width=0.45\textwidth]{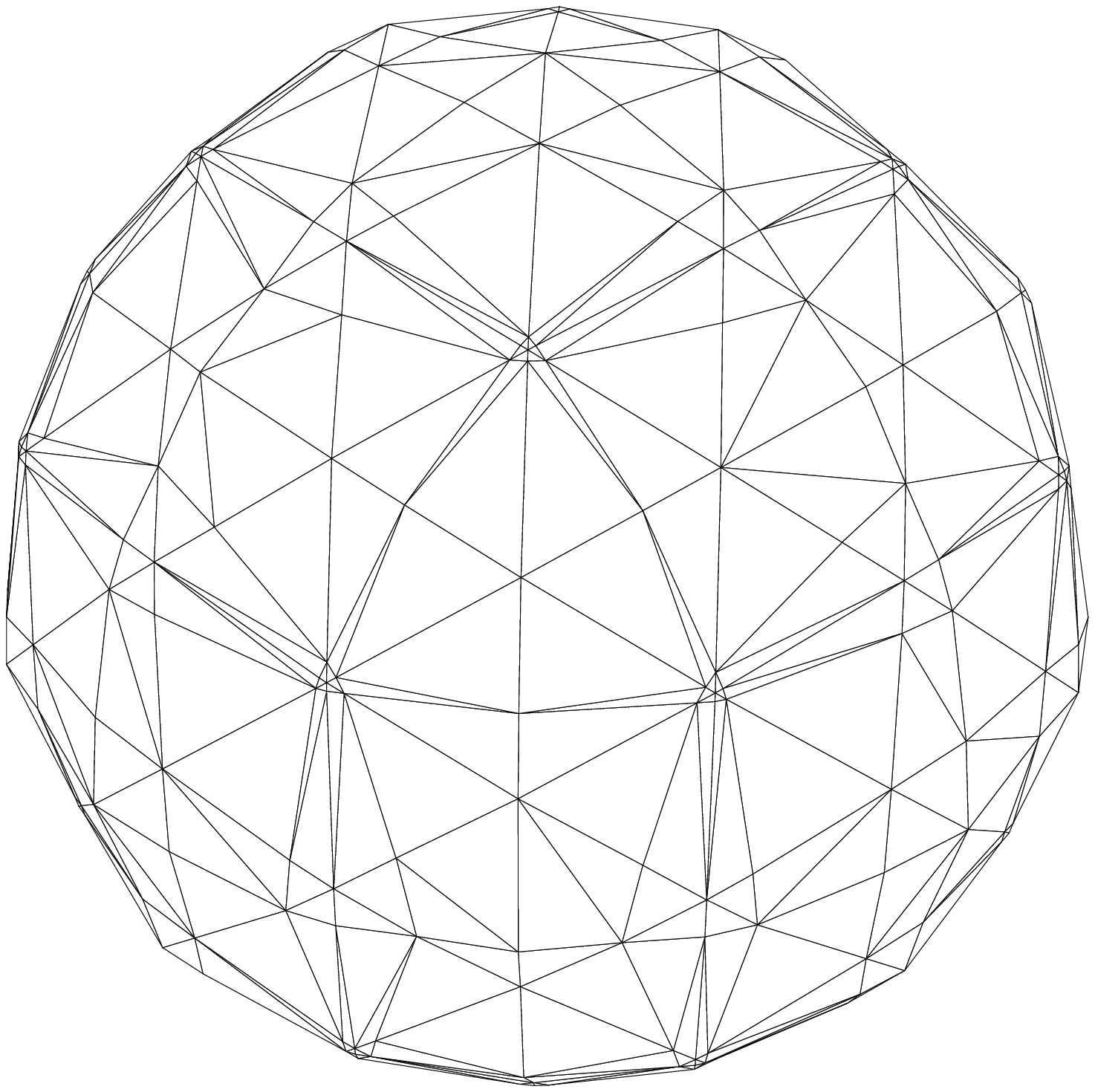}
  \includegraphics[width=0.45\textwidth]{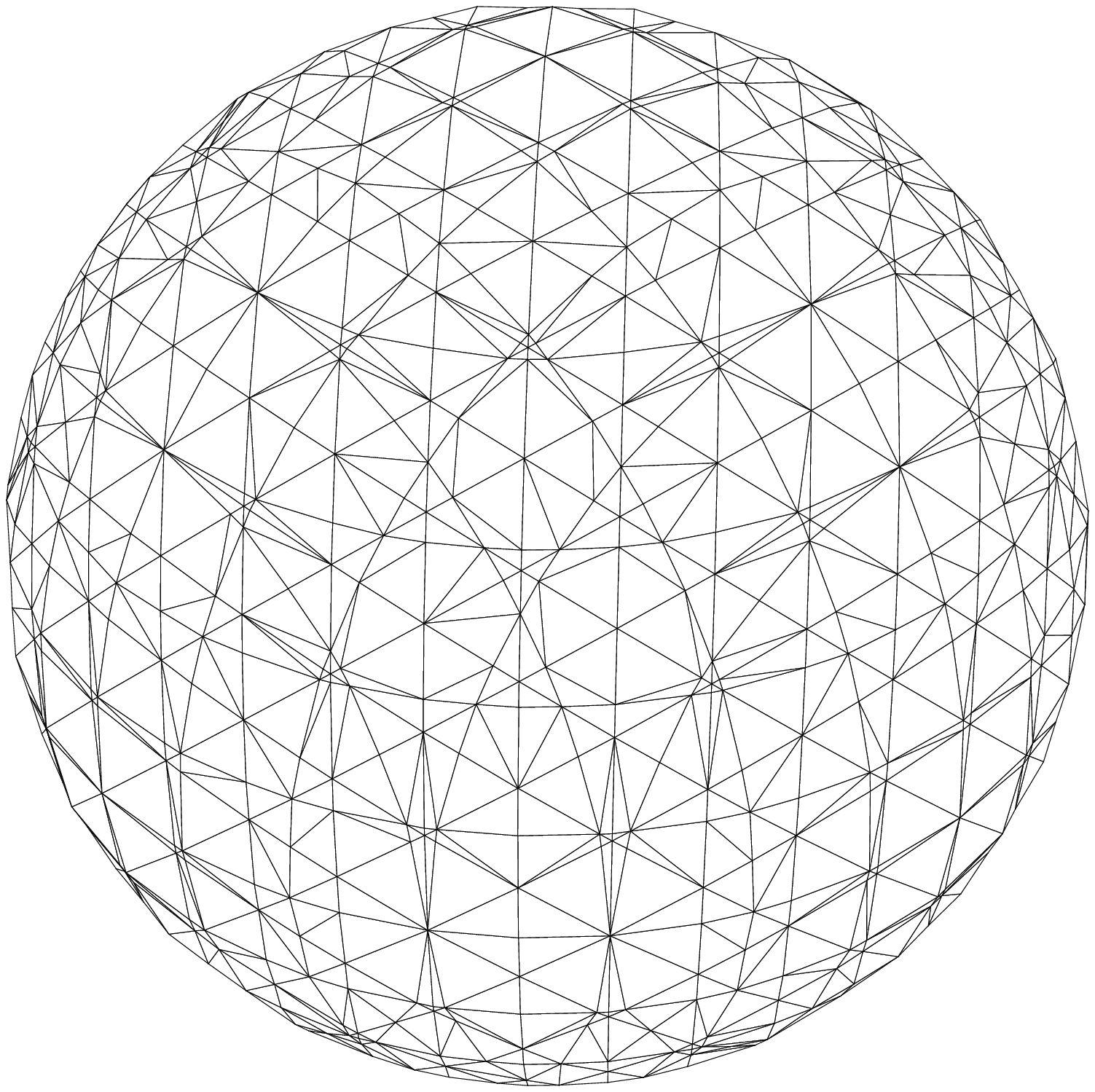}
\caption{Approximate interface $\Gamma_h$ for  a sphere, resulting from a coarse tetrahedral triangulation (left) and after one refinement (right).}
\label{fig:tri}
\end{center}
\end{figure}

Let $\mathcal{F}_h$ be the set of all  triangular segments $T$, then $\Gamma_h$ can
be decomposed as
\begin{equation} \label{defgammah}
 \Gamma_h=\bigcup\limits_{T\in\mathcal{F}_h} T.
\end{equation}
Note that the triangulation $\mathcal{F}_h$  is {not} necessarily  regular, i.e. elements from $T$ may have
very small internal angles and the size of neighboring triangles can vary strongly, cf.~Figure~\ref{fig:tri}.
In applications with level set functions (that represent $\Gamma$ implicitly), the approximation $\Gamma_h$ can be obtained as the zero level of a piecewise linear finite element  approximation of the level set function on the tetrahedral triangulation $\mathcal{T}_h$.

The surface finite element space is \textit{the space of traces on $\Gamma_h$ of all piecewise linear continuous functions with respect to the outer triangulation $\mathcal{T}_h$}.
This can be formally defined as follows.
We define a subdomain that contains $\Gamma_h$:
\begin{equation} \label{defomeg}
 \omega_h= \bigcup_{T \in \mathcal{F}_h} S_T,
\end{equation}
an a corresponding volume mesh finite element space
\begin{equation}
 V_h:=\{v_h\in C(\omega_h)\ |\ v_h|_{S_T}\in P_1~~ \forall\ T \in\mathcal{F}_h\},\lbl{e:2.6}
\end{equation}
where $P_1$ is the space of polynomials of degree one. $V_h$ induces the following
space on $\Gamma_h$:
\begin{equation}
 V_h^{\Gamma}:=\{\psi_h\in H^1(\Gamma_h)\ |\ \exists ~ v_h\in V_h\  \text{such that }\ \psi_h=v_h|_{\Gamma_h}\}.
\lbl{e:fem-space}
\end{equation}
When $c=0$, we require that any function $v_h$ from $V_h^{\Gamma}$ satisfies  $\int_{\Gamma_h}v_h \, \mathrm{d}\mathbf{s}=0$.
Given the surface finite element space $V_h^{\Gamma}$,  the finite element
discretization of \eqref{e:2.5} is as follows:   Find $u_h\in V_h^{\Gamma}$ such that
\begin{equation}
\eps\int_{\Gamma_h}\nath u\cdot\nath v\, \ds + \int_{\Gamma_h}(\bw^e\cdot\nath u) v\, \ds
 + \int_{\Gamma_h}c \,uv\,\ds=\int_{\Gamma_h}f^e v\, \ds \lbl{plainFEM}
\end{equation}
for all $v_h\in V_h^{\Gamma}$. Here $\mathbf{w}^e$ and $f^e$ are
the extensions  of $\wt$ and $f$, respectively, along normals to $\Gamma$ (the precise definition is given in the next section). Similar to the plain Galerkin finite element for  advection-diffusion equations, the method \eqref{plainFEM} is unstable unless the mesh is sufficiently fine such that the mesh Peclet number is less than one.

We introduce the following stabilized finite element method based on the standard SUPG approach, cf. \cite{TobiskaBook}:
 Find $u_h\in V_h^{\Gamma}$ such that
\begin{equation}
 a_h(u_h,v_h)=f_h(v_h)\quad \forall~ v_h\in V_h^{\Gamma}, \lbl{e:2.8}
\end{equation}
with
\begin{align}
a_h(u,v):=& \eps\int_{\Gamma_h}\nath u\cdot\nath v \, \ds
 + \int_{\Gamma_h}c \, uv\ds \nonumber \\
 &+ \frac12\left[\int_{\Gamma_h}(\bw^e\cdot\nath u) v \, \ds -\int_{\Gamma_h}(\bw^e\cdot\nath v) u\,  \ds \right] \label{eqah}\\
  &+\sum_{T\in\Fh}\delta_T\int_{T}(-\eps\Delta_{\Gamma_h}u + \bw^e\cdot\nath u + c\, u)\bw^e\cdot\nath v\, \ds, \nonumber \\
f_h(v):=&\int_{\Gamma_h}f^e v\ds + \sum_{T\in\Fh}\delta_T\int_{T}f^e(\bw^e\cdot\nath v)\, \ds. \label{eqfh}
 \end{align}
The stabilization parameter  $\delta_T$ depends on $T \subset S_T$. The diameter of the tetrahedron $S_T$ is denoted by $h_{S_T}$.  Let $\displaystyle \mathsf{Pe}_T:=\frac{h_{S_T} \|\mathbf{w}^e\|_{L^\infty(T)}}{2\eps}$
be the cell Peclet number.
We take
\begin{equation}
 \widetilde{\delta_T}=
\left\{
\begin{aligned}
&\frac{\delta_0 h_{S_T}}{\|\mathbf{w}^e\|_{L^\infty(T)}} &&\quad \hbox{ if } \mathsf{Pe}_T> 1,\\
&\frac{\delta_1 h^2_{S_T}}{\eps}  &&\quad \hbox{ if } \mathsf{Pe}_T\leq 1,
\end{aligned}
\right.\quad\text{and} \quad\delta_T=\min\{\widetilde{\delta_T},c^{-1}\}, \lbl{e:2.10}
\end{equation}
with some given positive constants  $\delta_0,\delta_1\geq 0$.

Since $u_h \in V_h^{\Gamma}$ is linear on every $T$ we have
$\Delta_{\Gamma_h} u_h =0$ on $T$, and thus
$a_h(u_h,v_h)$ simplifies to
\begin{multline}  \label{eqqr}
 a_h(u_h,v_h)=\eps\int_{\Gamma_h}\nath u_h\cdot\nath v_h \, \ds
+ \frac12\left[\int_{\Gamma_h}(\bw^e\cdot\nath u_h) v_h  -(\bw^e\cdot\nath v_h) u_h \, \ds \right] \\
+\int_{\Gamma_h}c \, u_h (v_h+\delta(\bfx) \bw^e\cdot\nath v_h)\, \ds+ \int_{\Gamma_h}\delta(\bfx)(\bw^e\cdot\nath u_h) (\bw^e\cdot\nath v_h) \, \ds,
\end{multline}
where $\delta(\bfx)=\delta_T$ for $\bfx\in T$.

\section{Error analysis} \label{sectanalysis}
The analysis in this section is organized as follows. First we collect some definitions and useful results in section~\ref{sectprelim}.
In section~\ref{sectcoercive}, we derive a coercivity result. In section~\ref{sectinter}, we present interpolation error bounds. In  sections~\ref{sectcont} and \ref{sectconsistency}, continuity and consistency results are derived.
Combining these analysis we obtain the  finite element  error bound given in section~\ref{sectmain}.
In the error analysis we use the following mesh-dependent norm:
\begin{equation} \label{defn}
 \| u \|_{\ast}:=\left(\eps\int_{\Gamma_h}|\nath u|^2\, \ds + \int_{\Gamma_h}\delta(\bfx)|\bw^e\cdot\nath u|^2\, \ds + \int_{\Gamma_h}c \,|u|^2\, \ds \right)^{\frac{1}{2}}.
\end{equation}
Here and in the remainder $|\cdot |$ denotes the Euclidean norm for vectors and the corresponding spectral norm
for matrices.

\subsection{Preliminaries}  \label{sectprelim}

For the hypersurface $\Gamma$, we define its $h$-neighborhood:
\begin{equation}
 U_h:=\{\mathbf{x}\in \mathbb{R}^3\ |\ \mathrm{dist}(\mathbf{x},\Gamma)<c_0 h\},\lbl{e:3.1}
\end{equation}
and assume  that $c_0$ is sufficiently large such that $\omega_h\subset U_h$ and $h$ sufficiently small such that
\begin{equation}
 5c_0 h<\left(\mathrm{max}_{i=1,2} \|\kappa_i\|_{L^{\infty}(\Gamma)}\right)^{-1}\lbl{e:3.2}
\end{equation}
holds, with $\kappa_i$ being the principal curvatures of $\Gamma$. Here and in what follows $h$ denotes
the maximum diameter for  tetrahedra of outer triangulation: $h=\max\limits_{S\in \omega_h} \text{diam}(S)$.

 Let $d : U_h\rightarrow \mathbb{R}$ be the signed
distance function, $|d(\mathbf{x})|=\mathrm{dist}(\mathbf{x},\Gamma)$ for all $\mathbf{x}\in U_h$. Thus
$\Gamma$ is the zero level set of $d$. We assume $d<0$ in the interior of $\Gamma$ and $d>0$ in the exterior and define $\mathbf{n}(\mathbf{x}):=\nabla d(\mathbf{x})$ for all $\mathbf{x}\in U_h$. Hence,  $\mathbf{n}=\mathbf{n}_{\Gamma}$ on $\Gamma$ and $|\mathbf{n}(\bfx)|=1$ for all $\bfx\in U_h$.  The Hessian of $d$ is denoted by
\begin{equation}
 \mathbf{H}(\bx):= \nabla^2 d(\bfx)\in \bbR^{3\times 3},\quad \bfx\in U_h.
\end{equation}
The eigenvalues of $\mathbf{H}(\bfx)$ are denoted by $\kappa_1(\bfx)$, $\kappa_2(\bfx)$, and~$0$.
 For $\bx\in\Gamma$ the eigenvalues $\kappa_i, i=1,2,$ are the principal curvatures.

For each $\bfx\in U_h$,  define the projection $\bfp: U_h\rightarrow\Gamma$ by
\begin{equation}
 \bfp(\bfx)=\bfx-d(\bfx)\bfn(\bfx).
\end{equation}
Due to the assumption \eqref{e:3.2},  the decomposition $\bfx=\bfp(\bfx)+d(\bfx)\bfn(\bfx)$ is unique.
We will need the orthogonal projector
\begin{equation*}
 \bfP(\bfx):=\mathbf{I}-\bfn(\bfx)\bfn(\bfx)^T, \quad \hbox{for }\bfx\in U_h.
\end{equation*}
Note that $\bn(\bx)= \bn(\bfp(\bx))$ and $\bfP(\bfx)=\bfP(\bfp(\bfx))$ for $\bx \in U_h$ holds.
The tangential derivative can be written as $\nabla_{\Gamma} g(\bfx)=\bfP\nabla g(\bfx)$ for $\bfx\in \Gamma$.
One can verify that for this projection and for the Hessian $\mathbf H$ the relation
$\mathbf H\bfP=\bfP\mathbf{H}=\mathbf{H}$ holds.
Similarly, define
\begin{equation}
 \bfP_h(\bfx):=\mathbf{I}-\bfn_{\Gamma_h}(\bfx)\bfn_{\Gamma_h}(\bfx)^T, \quad \hbox{for }\bfx\in \Gamma_h,~\bfx \hbox{ is not on an edge},
\end{equation}
where $\bfn_{\Gamma_h}$ is the unit (outward pointing) normal at $\bfx\in \Gamma_h$ (not on an edge). The tangential
derivative along $\Gamma_h$ is given by $\nabla_{\Gamma_h} g(\bfx)=\bfP_h(\bfx)\nabla g(\bfx)$  (not on an edge).
\\
\begin{Assumption} \label{ass1} \rm  In this paper, we assume that for all $T\in \mathcal{F}_h$:
\begin{align}
&\mathrm{ess\ sup}_{\bfx\in T}|d(\bfx)| \le c_1 h_{S_T}^2, \lbl{e:3.9}\\
&\mathrm{ess\ sup}_{\bfx\in T}|\bfn(\bfx)-\bfn_{\Gamma_h}(\bfx)| \le c_2   h_{S_T},\lbl{e:3.10}
\end{align}
where $h_{S_T}$ denotes the  diameter of the tetrahedron $S_T$ that contains $T$, i.e., $T=S_T \cap \Gamma_h$ and constants $c_1$, $c_2$ are independent of $h$, $T$.
\end{Assumption}
\ \\[1ex]
The assumptions \eqref{e:3.9} and \eqref{e:3.10} describe how accurate the piecewise planar approximation $\Gamma_h$ of $\Gamma$ is. If $\Gamma_h$ is constructed as the zero level of a piecewise linear interpolation of a level set function that characterizes $\Gamma$ (as in Fig.~\ref{fig:tri}) then these assumptions are fulfilled, cf. Sect. 7.3 in \cite{GrossReuskenBook}.

In the remainder, $A\lesssim B $ means $A\leq \tilde c B $ for some positive constant $\tilde c$ independent of $h$ and of the problem parameters $\eps$ and $c$.
$A\simeq B$ means that both $A\lesssim B $ and $B\lesssim A $.

For $\bfx\in\Gamma_h$, define
\begin{equation*}
 \mu_h(\bfx) = (1-d(\bfx)\kappa_1(\bfx))(1-d(\bfx)\kappa_2(\bfx))\bfn^T(\bfx)\bfn_h(\bfx).
\end{equation*}
The surface measures  $\ds$ and $\ds_{h}$ on $\Gamma$ and $\Gamma_h$, respectively, are related by
\begin{equation}
 \mu_h(\bfx)\ds_h(\bfx)=\ds(\bfp(\bfx)),\quad \bfx\in\Gamma_h. \lbl{e:3.16}
\end{equation}
The assumptions \eqref{e:3.9} and \eqref{e:3.10} imply that
\begin{equation}
 \mathrm{ess\ sup}_{\bfx\in\Gamma_h}(1-\mu_h)\lesssim h^2,\lbl{e:3.17}
\end{equation}
cf. (3.37) in \cite{Reusken08}.
The solution of  \eqref{e:2.3} is defined on $\Gamma$,
while its finite element approximation $u_h \in V_h^\Gamma$ is defined on $\Gamma_h$.
We need a suitable extension of a function from $\Gamma$ to its neighborhood. For a function $v$ on $\Gamma$ we define
\begin{equation}
 v^e(\bfx):= v(\bfp(\bfx)) \quad \hbox{for all } \bfx\in U_h.
\end{equation}
The following formula for this lifting function are known (cf. section 2.3 in \cite{Demlow06}):
\begin{align}
 \nabla u^e(\mathbf{x}) &= (\mathbf{I}-d(\mathbf{x})\mathbf{H})\nabla_{\Gamma} u(\bfp(\bfx)) \quad \hbox{ a.e. on } U_h,\label{grad1}\\
 \nabla_{\Gamma_h} u^e(\bfx) &= \bfP_h(\bfx)(\mathbf{I}-d(\mathbf{x})\mathbf{H})\nabla_{\Gamma} u(\bfp(\bfx)) \quad \hbox{ a.e. on } \Gamma_h,
\end{align}
with $\mathbf{H}=\mathbf{H}(\mathbf{x})$.
By direct computation one derives the relation
\begin{multline}
\nabla^2 u^e(\bfx)=(\bfP -d(\bfx) \mathbf{H})\nabla^2_{\Gamma} u(\bfp(\bfx))(\bfP -d(\bfx) \mathbf{H})-(n^T\nabla_{\Gamma} u(\bfp(\bfx))\mathbf{H}\\
-(\mathbf H\nabla_{\Gamma} u(\bfp(\bfx))) \bfn^T-\bfn (\mathbf H\nabla_{\Gamma} u(\bfp(\bfx)))^T -d\nabla_{\Gamma} \mathbf H:\nabla_{\Gamma} u(\bfp(\bfx)).
 \lbl{e:3.14n}
\end{multline}
For  sufficiently smooth $u$ and $|\mu| \leq 2$, using this relation  one obtains the estimate
 \begin{equation}
  |D^{\mu} u^e(\bfx)|\lesssim \left(\sum_{|\mu|=2}|D_{\Gamma}^{\mu} u(\bfp(\bfx))| + |\nabla_{\Gamma} u(\bfp(\bfx))|\right)\quad \hbox{a.e. on } U_h,
\lbl{e:3.13}
 \end{equation}
(cf. Lemma 3 in \cite{Dziuk88}). This further leads to (cf. Lemma 3.2 in \cite{Reusken08}):
 \begin{equation}
 \|D^{\mu} u^e\|_{L^2(U_h)}\lesssim\sqrt{h}\|u\|_{H^2(\Gamma)}, \quad |\mu|\leq 2. \lbl{e:3.14}
 \end{equation}

The next lemma is needed for the analysis in the following section.

\begin{lemma}\label{lem_div}  The following holds:
\[
\|\divth\bw^e\|_{L^\infty(\Gamma_h)}\lesssim h \|\nabla_{\Gamma}\bw\|_{L^\infty(\Gamma)}.
\]
\end{lemma}
\begin{proof} We use the following representation for the tangential divergence:
\begin{equation}\label{e:Marz}
\Div_\Gamma\bw(\bx)=\operatorname{tr}(\nabla_\Gamma\bw(\bx))= \operatorname{tr}(\bfP\nabla\bw(\bx)).
\end{equation}
Take $\bx\in\Gamma_h$, not lying on an edge. Using  \eqref{grad1} we obtain
\begin{align*}
 & \divth\bw^e(\bx) \\ & =\operatorname{tr}(\bfP_h\nabla\bw^e(\bx))=
\operatorname{tr}\left(\bfP_h(\mathbf{I}-d(\mathbf{x})\mathbf{H})\nabla_{\Gamma} \bw(\bfp(\bfx))\right)\\&=
\operatorname{tr}\left(\bfP\nabla_{\Gamma} \bw(\bfp(\bfx))\right)
+\operatorname{tr}\left((\bfP_h-\bfP)\nabla_{\Gamma} \bw(\bfp(\bfx))\right)-
d(\mathbf{x})\operatorname{tr}\left(\bfP_h\mathbf{H}\nabla_{\Gamma} \bw(\bfp(\bfx))\right).
\end{align*}
The first term vanishes due to
$
\operatorname{tr}\left(\bfP\nabla_{\Gamma} \bw(\bfp(\bfx))\right)=
\Div_\Gamma\bw(\bfp(\bfx))=0$.
The second and the third term can be bounded using \eqref{e:3.9}, \eqref{e:3.10}:
\[
|\bfP_h-\bfP|\lesssim h,\quad |d(\mathbf{x})\bfP_h\mathbf{H}|\lesssim h^2.
\]
This proves the lemma.
\end{proof}

\subsection{Coercivity analysis} \label{sectcoercive}
In the next lemma we present a coercivity result. We use the norm introduced in \eqref{defn}.

\begin{lemma} \label{thm0}
The following holds:
\begin{equation}\lbl{coercivity}
 a_h(v_h,v_h)\geq \frac{1}{2}\|v_h\|_{\ast}^2 \quad \text{for all}~~v_h \in V_h^\Gamma.
\end{equation}
\end{lemma}
\begin{proof}
For any $v_h\in V_h^{\Gamma}$, we have
\begin{equation}\label{aux2}
 a_h(v_h,v_h)= \|v_h\|_{\ast}^2
+\int_{\Gamma_h}c\, \delta(\bfx)v_h(\bw^e\cdot\nath v_h)\, \ds.
\end{equation}

The choice of $\delta_T$, cf. \eqref{e:2.10}, implies $c \, \delta(\bfx) \le1$. Hence the last term in \eqref{aux2} can be  estimated as follows:
\begin{align*}
 & |\int_{\Gamma_h}c \, \delta(\bfx)v_h(\bw^e\cdot\nath v_h)\, \ds|\\ &  \le \frac12\int_{\Gamma_h} c \, v_h^2\, \ds + \frac12\int_{\Gamma_h}c \, \delta(\bfx)^2(\bw^e\cdot\nath v_h)^2\, \ds\le \frac12 \|v_h\|_{\ast}^2.
\end{align*}
This yields \eqref{coercivity}.
\end{proof}
\ \\
As a consequence of this result,  we obtain the well-posedness of the discrete problem \eqref{e:2.8}.

\subsection{Interpolation error bounds} \label{sectinter}
Let $I_h:C(\bar{\omega}_h)\rightarrow V_h$ be the nodal interpolation operator.
For any $u\in H^2(\Gamma)$ the surface finite element function
$(I_h u^e)|_{\Gamma_h}\in V_h^{\Gamma}$ is an interpolant
of $u^e$ in $V_h^{\Gamma}$.

 For any $u\in H^2(\Gamma)$ the following estimates hold~\cite{Reusken08}:
\begin{align}
\|u^e-(I_h u^e)|_{\Gamma_h}\|_{L^2(\Gamma_h)}&\lesssim h^2\|u\|_{H^2(\Gamma)}, \lbl{e:3.18}\\
\|\nath u^e-\nath(I_h u^e)|_{\Gamma_h}\|_{L^2(\Gamma_h)}&\lesssim h\|u\|_{H^2(\Gamma)}.\lbl{e:3.19}
\end{align}

Using these results we easily obtain an  interpolation error estimate in the $\|\cdot\|_{\ast}$- norm:
\begin{lemma}
 For any $u\in H^2(\Gamma)$ the following holds:
\begin{equation} \lbl{e:3.20}
\| u^e-(I_h u^e)|_{\Gamma_h} \|_{\ast} \lesssim h(\eps^{1/2}+h^{1/2}+c^{\frac12}h) \|u\|_{H^2(\Gamma)}.
\end{equation}
\end{lemma}
\begin{proof} Define $\varphi:=u^e-(I_h u^e)|_{\Gamma_h} \in H^1(\Gamma_h)$.
Using the definition \eqref{e:2.10} of $\delta(\bfx)$, we get
\begin{equation}\label{aux3} \begin{split}
 \int_{\Gamma_h}\delta(\bfx)|\bw^e\cdot\nath\varphi|^2\ds & =
\sum_{T\in\Fh}\int_{T}\delta_T|\bw^e\cdot\nath\varphi|^2\, \ds \\
 & \lesssim h\|\nath\varphi\|_{L^2(\Gamma_h)}^2\lesssim h^3 \|u\|_{H^2(\Gamma)}^2 .
\end{split} \end{equation}
The remaining two terms in  $\|u^e-(I_h u^e)\|_{\ast}$ are estimated in a straightforward way using  \eqref{e:3.18} and \eqref{e:3.19}. This and \eqref{aux3} imply the inequality \eqref{e:3.20}.
\end{proof}

The next lemma estimates the interpolation error on the edges of the surface triangulation. In the remainder, $\Eh$ denotes the set of all edges in the interface triangulation $\mathcal{F}_h$.


\begin{lemma}\lbl{prop:3.3}
For all $u\in H^2(\Gamma)$ the following holds:
\begin{equation}
 \left(\sum_{E\in\Eh}\int_{ E}(u^e-I_h u^e)|_{\Gamma_h}^2\, \ds\right)^{1/2}\lesssim h^{3/2} \|u\|_{H^2(\Gamma)}.\lbl{e:3.24}
\end{equation}
\end{lemma}
\begin{proof}
Define $\phi:= u^e-I_h u^e \in H^1(\omega_h)$.
Take $E\in\Eh$ and let $T \in \mathcal{F}_h$  be a corresponding planar segment of which $E$ is an edge. Let $W$ be a side of the tetrahedron $S_T$ such that $E \subset W$.
 From Lemma 3 in \cite{Hansbo02} we have
\[
  \|\phi\|_{L^2(E)}^2 \lesssim h^{-1} \|\phi\|_{L^2(W)}^2 + h \|\phi\|_{H^1(W)}^2.
\]
 From the standard trace inequality
\[
 \|w\|_{L^2(\partial S_T)}^2 \lesssim h^{-1} \|w\|_{L^2(S_T)}^2 + h \|w\|_{H^1(S_T)}^2 \quad \text{for all}~~ w \in H^1(S_T),
\]
applied to $\phi$ and $\partial_{x_i} \phi$, $i=1,2,3$, we obtain
\begin{align*}
 h^{-1} \|\phi\|_{L^2(W)}^2 & \lesssim h^{-2} \|\phi\|_{L^2(S_T)}^2 +  \|\phi\|_{H^1(S_T)}^2, \\
h \|\phi\|_{H^1(W)}^2& \lesssim \|\phi\|_{H^1(S_T)}^2 + h^2 \|u^e\|_{H^2(S_T)}^2 .
\end{align*}
 From standard error bounds for the nodal interpolation operator $I_h$ we get
\[
  \|\phi\|_{L^2(E)}^2\lesssim h^{-2} \|\phi\|_{L^2(S_T)}^2+  \|\phi\|_{H^1(S_T)}^2 + h^2 \|u^e\|_{H^2(S_T)}^2 \lesssim h^2 \|u^e\|_{H^2(S_T)}^2.
\]
Summing over $E\in\Eh$ and using $\|u^e\|_{H^2(\omega_h)} \lesssim h^\frac12 \|u\|_{H^2(\Gamma)}$, cf. \eqref{e:3.14}, results in
\[
 \sum_{E\in\Eh} \|\phi\|_{L^2(E)}^2 \lesssim  h^2 \|u^e\|_{H^2(\omega_h)}^2 \lesssim h^3 \|u\|_{H^2(\Gamma)}^2,
\]
which completes the proof. \end{proof}

\subsection{Continuity estimates} \label{sectcont}

In this section we derive a  continuity estimate for
the bilinear form $a_h(\cdot,\cdot)$. If one applies partial integration to the integrals that occur in $a_h(\cdot,\cdot)$ then \emph{jumps across the edges} $E \in \Eh$
occur. We start with a lemma that yields  bounds for such jump terms.
Related to these jump terms we introduce the following notation.
For each $T\in\Fh$, denote by ${\mbf{m}_h}|_{E}$ the outer normal to an edge $E$ in the plane which contains element $T$. Let $[\mbf{m}_h]|_{E}=\mbf{m}_h^+ +\mbf{m}_h^-$ be the jump of the outer normals to the edge in two neighboring elements, c.f. Figure \ref{fig:3.1}.
\begin{figure}[ht!]
\vspace*{-2mm}
    \centering
  \resizebox{!}{5.5cm}
    {\includegraphics[ height=60mm]{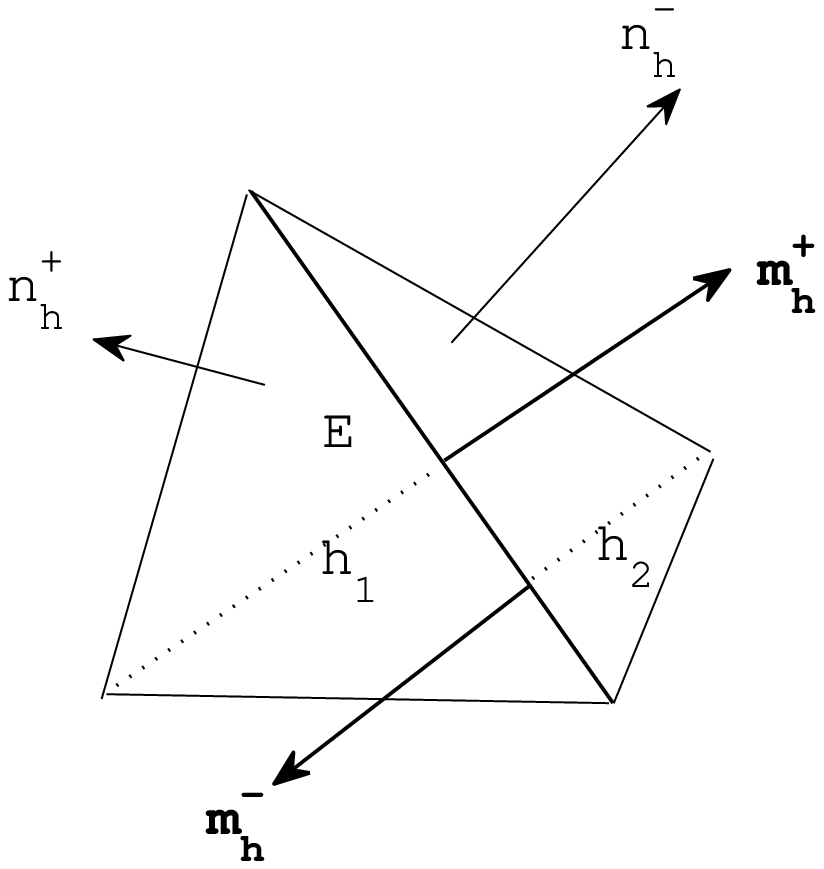}}
    \vspace*{-10mm}
    \caption{ }
   \lbl{fig:3.1}
 \end{figure}

\begin{lemma}\lbl{lem:jumpvel}
The following holds:
\begin{equation} \label{jump_est}
 |\bfP(\bx)[\mbf{m}_h](\bfx)|\lesssim h^2
\qquad a.e.\ \bfx\in E. 
\end{equation}
\end{lemma}
\begin{proof}
Let $E$  be the common side of two elements $T_1$ and $T_2$ in $\mathcal{F}_h$, and
$\bfn_h^+$, $\bfn_h^-$,
  $\mbf m _h^+$ and $\mbf m_h^-$ are the unit normals as illustrated in Figure~\ref{fig:3.1}.
 Denote by $\mbf s_h$ the unit (constant) vector along the common side $E$, which can be represented as $\mbf s_h=\bfn_h^{+}\times \mbf m_h^{+}= \mbf m_h^{-}\times\bfn_h^{-}$.  The jump across $E$ is given by
\[
[\mbf{m}_h]=\mbf s_h\times(\bfn_h^{+}-\bfn_h^{-}).
\]
For each $\bfx\in E$ and $\bfp(\bfx)\in\Gamma$, let $\bfn=\mbf{n}(\bfp(\bfx))$ be the unit normal to $\Gamma$ at $\bfp(\bfx)$ and $\bfP= \bfP(\bx)= \mathbf{I} - \mbf{n} \mbf{n}^T$ the corresponding orthogonal projection.
Using \eqref{e:3.10}, we get
\begin{equation*}
|\bfn_h^--\bfn_h^+ |\le |\bfn_h^+-\bfn|+ |\bfn_h^--\bfn| \lesssim h_{S_{T_1}}+h_{S_{T_2}} \lesssim h.
\end{equation*}
Since $|\bfn_h^-|=|\bfn_h^+|=|\bfn|=1$, the above estimate implies
\[
\bfn_h^+-\bfn_h^- =c h^2 \bfn + \be_1,\quad \be_1\perp \bfn,\quad |\be_1|\lesssim h.
\]
We also have
\[
\mbf s_h=\bfn_h^{+}\times \mbf m_h^{+}=(\bfn+(\bfn_h^{+}-\bfn))\times \mbf m_h^{+}=
\bfn\times \mbf m_h^{+}+\be_2,\quad |\be_2|\lesssim h.
\]
We use the decomposition
\[
 \bfP [\mbf{m}_h]= \bfP \left[(\bfn\times \mbf m_h^{+}+\be_2) \times \left( c h^2 \bfn + \be_1\right)\right].
\]
Since $\be_1\perp \bfn$ we have $(\bfn\times \mbf m_h^{+}) \times \be_1\parallel \bfn$ and thus $\bfP \big((\bfn\times \mbf m_h^{+}) \times \be_1\big)=0$.
Therefore, we get
\begin{equation}\label{aux20}
|\bfP [\mbf{m}_h]|\lesssim h^2+|\be_1|\ |\be_2|\lesssim h^2,
\end{equation}
i.e., the result \eqref{jump_est} holds.
\end{proof}
\ \\[1ex]
In the analysis below, we need an inequality of the form $\|v_h\|_{L^2(\Gamma_h)} \lesssim \|v_h\|_\ast$ for all $v_h\in V_h^\Gamma$. This result can be obtained as follows. First we consider the case $c=0$. Then the functions $v_h \in  V_h^\Gamma$ satisfy $\int_{\Gamma_h} v_h \, \ds =0$. We assume that in $V_h^\Gamma$ a discrete analogon of the Friedrich's inequality \eqref{FdrA} holds uniformly with respect to $h$, i.e., there exists a constant $C_F$ independent of $h$ such that
\begin{equation} \label{FrB}
 \|v_h\|_{L^2(\Gamma_h)}^2 \leq C_F \|\nabla_{\Gamma_h} v_h\|_{L^2(\Gamma_h)}^2 \quad \text{for all}~~v_h \in V_h^\Gamma.
\end{equation}
Now we reduce the parameter domain $\eps \in (0,1],~c>0$ as follows. For a given generic constant $c_0$ with $0< c_0 <1$, in the remainder  we restrict to the parameter set
\begin{equation} \label{pardomain}
\eps \in (0,1],~~c \in \{0\} \cup [ c_0 \eps, \infty).
\end{equation}
For $c>0$ we then have
\begin{align*}
 c_0 \|v_h\|_{L^2(\Gamma_h)}^2 & \leq \frac{2 c_0}{c_0+1} \frac{1}{c} \|c^\frac12 v_h\|_{L^2(\Gamma_h)}^2 \leq \frac{2}{c/c_0 + c} \|v_h\|_\ast^2 \leq \frac{2}{\eps +c} \|v_h\|_\ast^2,
\end{align*}
and combing this with the result in \eqref{FrB} for the case $c=0$ we get
\begin{equation} \label{basic}
 \|v_h\|_{L^2(\Gamma_h)} \lesssim \frac{1}{\sqrt{\eps + c}} \|v_h\|_\ast \quad \text{for all}~~v_h\in V_h^\Gamma
\end{equation}
and arbitrary $\eps \in (0,1],~c \in \{0\} \cup [ c_0 \eps,  \infty)$.\\

In the proof of Theorem~\ref{thm2} below we need a bound for $\|v_h\|_{L^2(\mathcal{E}_h)}$ in terms of $\|v_h\|_\ast$. Such a result is derived with the help of the following lemma.

\begin{lemma}\label{lem_edge_est}
Assume the outer tetrahedra mesh size satisfies $h\le h_0$,  with some sufficiently small $h_0\simeq 1$,
depending only on the constant $c_2$ from \eqref{e:3.10}. The following holds:
\begin{equation} \label{edge_est}
 \sum_{E\in \Eh}\int_{E}v_h^2 \, \ds  \lesssim  h^{-1} \|v_h\|^2_{L^2(\Gamma_h)}+
  h \|\nabla_{\Gamma_h} v_h\|^2_{L^2(\Gamma_h)} \quad \text{for all}~~v_h \in V_h^\Gamma
\end{equation}
\end{lemma}
\begin{proof}
Let $E\in \Eh$  be an edge  of a triangle $T\in\mathcal{F}_h$ and $S_T\in\mathcal{T}_h$ is the corresponding tetrahedron of the outer triangulation. Consider the patch $\widetilde{\omega}(S_T)$ of all $S\in\mathcal{T}_h$
touching $S_T$. Denote  $\omega(S_T)=\widetilde{\omega}(S_T)\cap\Gamma_h$. Let $v_h$ be an arbitrary fixed function
from $ V_h^\Gamma$.
We shall prove the bound
\begin{equation}\label{local_est}
\int_{E}v_h^2 \, \ds  \lesssim  h^{-1} \|v_h\|^2_{L^2(\omega(S_T))}+
  h \|\nabla_{\Gamma_h} v_h\|^2_{L^2(\omega(S_T))}.
\end{equation}
Then summing over all $E\in \Eh$ and using that $\omega(S_T)$ consists of a uniformly bounded number of
tetrahedra (due to the regularity of the outer mesh), we obtain  \eqref{edge_est}.
\smallskip

Let $\mathbb{P}$ be a plane containing $T$. We can define for sufficiently small $h$ an injective  mapping $\phi:\omega(S_T)\to\mathbb{P}$  such that $|\nabla\phi|\lesssim1$ and  $|\nabla(\phi^{-1})|\lesssim1$.  For example, $\phi$ can be build by the orthogonal projection
on $\mathbb{P}$. Then $|\nabla\phi|\lesssim 1$ and $|\nabla(\phi^{-1})(\bx)|\lesssim (\sin\alpha)^{-1}$,
where $\alpha$ is the angle between $\mathbb{P}$ and $\bn_h(\phi^{-1}(\bx))$. Due to assumption \eqref{e:3.10} we have $1\lesssim \sin\alpha$ for sufficiently small $h$.   If $\phi$ is the orthogonal projection on $\mathbb{P}$, then $\phi(E)=E$.  Thus we get
\begin{equation}\label{equiv}
\left\{
\begin{split}
\int_{E}v_h^2\ds&=\int_{\phi(E)}(v_h\circ\phi^{-1})^2\ds,\\
 \|v_h\|_{L^2(\omega(S_T))}&\simeq\|v_h\circ\phi^{-1}\|_{L^2(\phi(\omega(S_T)))},\\
 \|\nabla_{\Gamma_h} v_h\|_{L^2(\omega(S_T))}&\simeq\|\nabla_{\mathbb{P}}( v_h\circ\phi^{-1})\|_{L^2(\phi(\omega(S_T)))}.
 \end{split}\right.
\end{equation}

Due to the shape regularity of $S\in\widetilde{\omega}(S_T)$ we have
\[h\lesssim\text{dist}(E,\partial\,\widetilde{\omega}(S_T))\le\text{dist}(E,\partial\,\omega(S_T)).\]
Hence, from $|\nabla\phi^{-1}|\lesssim1$ it follows that
$h\lesssim \text{dist}(\phi(E),\partial\,\phi(\omega(S_T)))$. Thus, we may consider a rectangle $Q\subset\phi(\omega(S_T))$
such that $E=\phi(E)$ is a side of $Q$ and $|Q|\simeq h|E|$. By the standard trace theorem and scaling argument we
get
\[
\int_{\phi(E)}(v_h\circ\phi^{-1})^2\ds \lesssim h^{-1}\|v_h\circ\phi^{-1}\|^2_{L^2(Q)}+ h\|\nabla_{\mathbb{P}} (v_h\circ\phi^{-1})\|^2_{L^2(Q)}.
\]
This together with \eqref{equiv} and $Q\subset\phi(\omega(S_T))$ implies \eqref{local_est}.
\end{proof}

An immediate consequence of the lemma and \eqref{basic} is the following corollary.

\begin{corollary}\lbl{lem:3.2}
The following estimate holds:
\begin{equation}\label{aux1}
h\sum_{E\in \Eh}\int_{E}v_h^2\, \ds  \lesssim \big( \frac{1}{\eps +c} + \frac{h^2}{\eps}\big) \|v_h\|_\ast^2 \quad \text{for all}~~v_h \in V_h^\Gamma.
\end{equation}
\end{corollary}

We are  now in position to prove a  continuity result for the surface finite element bilinear form.

\begin{lemma} \label{thm2}
 For any $u\in H^2(\Gamma)$ and $v_h\in V_h^{\Gamma}$, we have
\begin{equation}
 |a_h(u^e-(I_h u^e)|_{\Gamma_h},v_h)|\lesssim \left(\eps^{1/2}+h^{1/2}+c^{\frac12}h+\frac{h^2}{\sqrt{\eps+c}}+ \frac{h^3}{\sqrt{\eps}}\right)h
 \|u\|_{H^2(\Gamma)} \|v_h \|_{\ast}.
\end{equation}
\end{lemma}
\begin{proof} Define $\phi=u^e-(I_h u^e)|_{{\Gamma_h}}$, then
\begin{equation} \label{aux9}
 \begin{split}
a_h(\phi,v_h) & = \eps\int_{\Gamma_h}\nath \phi\cdot\nath v_h \ds\\
 & + \int_{\Gamma_h} \frac12\big((\bw^e\cdot\nath \phi) v_h-(\bw^e\cdot\nath v_h) \phi\big) +c\,\phi v_h \,\ds \\
 & + \sum_{T\in\Fh}  \delta_T \int_T \left(-\eps\Delta_{\Gamma_h}\phi
 +  \bw^e\cdot\nath \phi + c\phi \right)\bw^e\cdot\nath v_h \, \ds.
\end{split}
\end{equation}
 We estimate  $a_h(\phi,v_h)$ term by term.
Due to \eqref{e:3.18}, \eqref{e:3.19}, we get for the first term on  the righthand side of \eqref{aux9}:
\begin{equation}\label{aux8}
\Big|\eps\int_{\Gamma_h}\nath \phi\cdot\nath v_h \ds\Big|\lesssim \eps h\|u\|_{H^2(\Gamma)}\|\nath v_h\|_{L^2(\Gamma_h)}
\le \eps^{\frac12}h\|u\|_{H^2(\Gamma)}\|v_h \|_{\ast}.
\end{equation}
To the second term on the righthand side of \eqref{aux9} we apply integration  by parts:
\begin{equation}\label{aux10}
 \begin{split} &   \int_{\Gamma_h}\frac12\big((\bw^e\cdot\nath \phi) v_h-(\bw^e\cdot\nath v_h) \phi\big) + c \phi v_h \,\ds \\
 & =\int_{\Gamma_h}c \phi v_h\ds-\int_{\Gamma_h}(\bw^e\cdot\nath v_h) \phi\ds \\
& + \frac12 \sum_{T\in\Fh}\int_{\pa T}(\bw^e\cdot \mathbf{m}_h)\phi v_h\ds -\frac12\int_{\Gamma_h}(\divth\bw^e) \phi v_h\ds \\
 & =:I_1+I_2+I_3+I_4.
\end{split}
\end{equation}
The term $I_1$ can be estimated by
\begin{equation*}
|I_1|\lesssim c^{\frac12}\|\phi \|_{L^2(\Gamma_h)}\|\sqrt{c}v_h\|_{L^2(\Gamma_h)} \lesssim h^2c^{\frac12}\|u \|_{H^2(\Gamma)}\|v_h\|_{\ast}.
\end{equation*}
To estimate $I_2$, we consider the advection-dominated case and the diffusion-dominated case separately. If $\mathsf{Pe}_T>1$, we have
\begin{equation*}
 \begin{aligned}
\int_{T} (\bw^e\cdot\nath v_h) \phi\ds&\lesssim   \delta_{T}^{-1/2}\|\phi\|_{L^2(T)}\left(\int_{T}\delta_T  (\bw^e\cdot\nath v_h)^2\ds\right)^{1/2}
\\
&\lesssim { \max(h^{-1/2},c^{1/2})}\|\phi\|_{L^2(T)}\left(\int_{T}\delta_T  (\bw^e\cdot\nath v_h)^2\ds\right)^{1/2},
 \end{aligned}
\end{equation*}
and if $\mathsf{Pe}_T\leq 1$:
\begin{equation*}
 \begin{aligned}
\int_{T} (\bw^e\cdot\nath v_h) \phi\ds&\lesssim  \|\bw^e\|_{L^{\infty}(T)} \|\nath v_h\|_{L^2(T)}\|\phi\|_{L^2(T)}
\\
&\lesssim \eps^{1/2} h^{-1} \|\eps^{1/2}\nath v_h\|_{L^2(T)}\|\phi\|_{L^2(T)} .
 \end{aligned}
\end{equation*}
 Summing over $T\in\Fh$  we obtain
\begin{equation*}
 |I_2|\lesssim (h^{-1/2}+\eps^{1/2}h^{-1}) \|\phi\|_{L^2(\Gamma_h)} \|v_h\|_{\ast}\lesssim h(c^{1/2}h+h^{1/2}+\eps^{1/2}) \|u\|_{H^2(\Gamma)} \|v_h\|_{\ast}.
\end{equation*}
The term $I_3$ is estimated using $\mathbf P \bw^e=\bw^e$, Lemmas~\ref{prop:3.3}, \ref{lem:jumpvel}, and Corollary~\ref{lem:3.2}:
\begin{equation}\label{aux22}
 \begin{aligned}
|I_3| &\lesssim \Big|\sum_{E\in\Eh}\int_{E}( \bw^e\cdot[\mbf{m}_h])\phi v_h\ds\Big| \\
&\lesssim  \left(\sum_{E\in\Eh}\int_{E}|\phi|^2\ds\right)^{1/2}
\left(\sum_{E\in\Eh}\int_{E} |\mathbf{P} [\mbf{m}_h] |^2 v_h^2 \, \ds\right)^{1/2}\\
&\lesssim h^{3}\|u\|_{H^2(\Gamma)}\left( h \sum_{E\in\Eh}\int_{E}v_h^2 \, \ds\right)^{1/2} \lesssim \left( \frac{h^3}{\sqrt{\eps +c}} +\frac{h^4}{\sqrt{\eps}}\right) \|u\|_{H^2(\Gamma)}\|v_h\|_\ast.
 \end{aligned}
\end{equation}
The term $I_4$ in \eqref{aux10} can be bounded due to Lemma~\ref{lem_div}, the interpolation
bounds and \eqref{basic}:
\[
\begin{aligned}
|I_4|&\le \frac12\|\divth\bw^e\|_{L^\infty(\Gamma_h)}\|\phi\|_{L^2(\Gamma_h)} \|v_h\|_{L^2(\Gamma_h)}\lesssim h^{3}\|u\|_{H^2(\Gamma)} \|v_h\|_{L^2(\Gamma_h)}\\&\le \frac{h^3}{\sqrt{\eps+c}} \|u\|_{H^2(\Gamma)} \|v_h\|_{\ast}.
\end{aligned}
\]
Finally we treat the third term on the righthand side of  \eqref{aux9}. Using  $\delta_T \|\bw^e\|_{L^\infty(T)} \lesssim h$, $\delta_T \eps \lesssim 1$,   $\delta_T c \leq 1$ and
the interpolation estimates \eqref{e:3.18} and \eqref{e:3.19} we obtain:
\begin{equation}\label{aux11}
 \begin{aligned}
  &\sum_{T\in\Fh}  \delta_T \int_T (-\eps\Delta_{\Gamma_h}\phi
 + \bw^e\cdot\nath \phi + c \phi ) \bw^e\cdot\nath v_h \, \ds \\
\lesssim &\left(\sum_{T\in\Fh} \delta_{T}\big(\eps^2 \|\Delta_{\Gamma_h}u^e\|^2_{L^2(T)}+\|\bw^e\cdot\nath\phi\|^2_{L^2(T)}+c^2\|\phi\|^2_{L^2(T)}\big)\right)^{1/2}\|v_h\|_{\ast}\\
\lesssim & ~ (\eps^{1/2}+h^{1/2}+c^{\frac12}h)h\|u\|_{H^2(\Gamma)}\|v_h\|_{\ast}.
 \end{aligned}
\end{equation}
Combing  the inequalities \eqref{aux8}--\eqref{aux11} proves the result of the lemma.
\end{proof}

\subsection{Consistency estimate} \label{sectconsistency}
The consistency error of the finite element method \eqref{e:2.8} is due to geometric errors resulting
from the approximation of $\Gamma$ by $\Gamma_h$. To estimate this geometric errors we need  a few additional
definitions and results, which can be found in, for example, \cite{Demlow06}. For $\bfx\in \Gamma_h$ define $\tilde{\bfP}_{h}(\bfx)= \mathbf{I}-\bfn_h(\bfx)\bfn(\bfx)^T/(\bfn_h(\bfx)\cdot\bfn(\bfx))$.
 One can represent the surface gradient of $u\in H^1(\Gamma)$ in terms of $\nath u^e$ as follows
\begin{equation} \label{hhl}
 \nat u(\bfp(\bfx))=(\mathbf{I}-d(\bfx)\mathbf{H}(\bfx))^{-1} \tilde{\bfP}_{h}(\bfx) \nath u^e(\bfx)~~ \hbox{ a.e. }\bfx\in \Gamma_h.
\end{equation}
Due to \eqref{e:3.16}, we get
\begin{equation}\label{aux13}
 \int_{\Gamma}\nat u\nat v\, \ds=\int_{\Gamma_h} \mathbf{A}_h\nath u^e\nath v^e \, \ds \quad \hbox{for all } v\in H^1(\Gamma),
\end{equation}
with $\mathbf{A}_h(\bfx)=\mu_h(\bfx) \tilde{\bfP}^T_h(\bfx)(\mathbf{I}-d(\bfx)\mathbf{H}(\bfx))^{-2}\tilde{\bfP}_h(\bfx)$.
 From $\bw \cdot \bn =0$ on $\Gamma$ and $\bw^e(\bx)=\bw(\bfp(\bfx))$, $\bn(\bx) =\bn(\bfp(\bfx))$ it follows that $\bn(\bx) \cdot \bw^e(\bx)=0$ and thus $\bw(\bfp(\bfx))= \tilde \bfP_h(\bx)\mathbf{w}^e(\bfx)$ holds. Using this, we get the relation
\begin{equation}\label{aux14}
 \int_{\Gamma}(\bw\cdot \nat u) v \, \ds =  \int_{\Gamma_h} (\mathbf{B}_h \mathbf{w}^e \cdot \nath u^e)v^e\, \ds,
\end{equation}
with $\mathbf{B}_h=\mu_h(\bfx)\tilde{\bfP}_h^T(I-d\mathbf{H})^{-1}\tilde\bfP_h$.
In the proof we use the lifting procedure $\Gamma_h \to \Gamma$  given by
\begin{equation} \label{deflift}
 v_h^l(\bfp(\bfx)): =v_h(\bfx) \quad \text{for}~~ \bfx\in\Gamma_h.
\end{equation}
It is easy to see that $v_h^l\in H^1(\Gamma)$.

The following lemma estimates the consistency error of the finite element method \eqref{e:2.8}.

\begin{lemma} \label{thm3}
 Let $u\in H^2(\Gamma)$ be the solution of \eqref{e:2.5}, then we have
\begin{equation}
 \sup_{v_h\in V_h^{\Gamma}}\frac{|f_h(v_h)-a_h(u^e,v_h)|}{\|v_h\|_{\ast}} \lesssim \big(h^{\frac12}+c^{\frac12}h+\frac{h}{\sqrt{c+\eps}}\big)\,h(\|u\|_{H^2(\Gamma)}+\|f\|_{L^2(\Gamma)}).
\end{equation}
\end{lemma}
\begin{proof}
The residual is decomposed as
\begin{equation} \label{er}
 f_h(v_h)-a_h(u^e,v_h)=f_h(v_h)-f(v_h^l)+a(u,v_h^l)-a_h(u^e,v_h).
\end{equation}
The following holds:
\begin{align*}
f(v^l_h)=&\,\int_{\Gamma}fv_h^l\, \ds= \int_{\Gamma_h}\mu_h f^e v_h\, \ds,\\
a(u,v^l_h)=&\,\eps \int_{\Gamma}\nabla_{\Gamma} u \nabla_{\Gamma} v_h^l\, \ds
+\int_{\Gamma}(\bw\cdot\nabla_{\Gamma} u) v_h^l\ds+\int_{\Gamma}cuv_h^l\, \ds\\
=&\,\eps \int_{\Gamma}\nabla_{\Gamma} u \nabla_{\Gamma} v_h^l\, \ds
+\frac12\int_{\Gamma}(\bw\cdot\nabla_{\Gamma} u) v_h^l-(\bw\cdot\nabla_{\Gamma} v_h^l) u\,\ds +\int_{\Gamma}cuv_h^l\, \ds\\
=&\, \eps \int_{\Gamma_h} \mathbf{A}_h\nabla_{\Gamma_h} u^e\nabla_{\Gamma_h} v_h\, \ds+ \frac12\int_{\Gamma_h} (\mathbf{B}_h \bw^e\cdot\nath u^e)v_h\, \ds\\
&\,-\frac12\int_{\Gamma_h}(\mathbf{B}_h \bw^e\cdot\nath v_h ) u^e\, \ds+\int_{\Gamma_h}\mu_h c u^ev_h\, \ds.
\end{align*}
Substituting these relations into \eqref{er}  and using \eqref{eqah}, \eqref{eqfh} results in
\begin{multline}\label{aux12}
  f_h(v_h)-a_h(u^e,v_h)=
  \int_{\Gamma_h}(1-\mu_h)f^ev_h\, \ds+\eps\int_{\Gamma_h}(\mathbf{A}_h-\bfP_h)\nabla_{\Gamma_h} u^e \cdot \nabla_{\Gamma_h}v_h\, \ds\\
+\frac12\int_{\Gamma_h}((\mathbf{B}_h-\bfP_h)\mathbf{w}^e\cdot\nath u^e) v_h\, \ds-\frac12\int_{\Gamma_h}((\mathbf{B}_h-\bfP_h)\mathbf{w}^e\cdot\nath v_h)  u^e\, \ds  \\ + \int_{\Gamma_h}(\mu_h-1)cu^e v_h\, \ds
+\sum_{T\in\Fh}\delta_T\int_T \big(f^e+\eps\Delta_{\Gamma_h}u^e-\mathbf{w}^e\cdot\nath u^e-cu^e \big) \bw^e\cdot\nath v_h \, \ds \\
=:I_1+I_2+I_3+I_4+I_5+\varPi_1.
\end{multline}
We estimate the $I_i$ terms separately.
Applying  \eqref{e:3.17} and \eqref{basic} we get
\begin{align}
I_1&\lesssim h^2 \|f^e\|_{L^2(\Gamma_h)}\|v_h\|_{L^2(\Gamma_h)}\lesssim \frac{h^2}{\sqrt{c +\eps }} \|f\|_{L^2(\Gamma)}\|v_h\|_{\ast},\label{aux15}\\
 I_5&\lesssim h^2 c^{\frac12}\|u^e\|_{L^2(\Gamma_h)}\|\sqrt{c}v_h\|_{L^2(\Gamma_h)}\lesssim h^2c^{\frac12} \|u\|_{L^2(\Gamma)}\|v_h\|_{\ast}.
\end{align}
One can show, cf. (3.43) in \cite{Reusken08}, the bound
\begin{equation*}
 |\bfP_h-\mathbf{A}_h|=|\bfP_h(\mathbf I-\mathbf A_h)|\lesssim h^2.
\end{equation*}
Using this  we obtain
\begin{align}
& I_2\lesssim \eps h^2 \|\nath u^e\|_{L^2(\Gamma_h)}\|\nath v_h\|_{L^2(\Gamma_h)} \lesssim \eps^{1/2}h^2 \|u^e\|_{H^2(\Gamma)}\|v_h\|_{\ast}.
\end{align}
Since $(\mathbf I-d\mathbf{H})^{-1}=\mathbf I+O(h^2)$, we also estimate
\begin{equation*}
 |\mathbf B_h-\bfP_h|\lesssim h^2+|\mathbf{A}_h-\bfP_h| \lesssim h^2.
\end{equation*}
This yields
\begin{equation}\label{aux16}
 I_3\lesssim  h^2 \|\nath u^e\|_{L^2(\Gamma_h)}\| v_h\|_{L^2(\Gamma_h)} \lesssim  \frac{h^2}{\sqrt{c+\eps }} \|u\|_{H^2(\Gamma)}\|v_h\|_{\ast}.
\end{equation}
To estimate $I_4$ we use  the definition \eqref{e:2.10} of $\delta_T$.
If $\displaystyle \mathsf{Pe}_T \leq 1$, then
$\eps^{-\frac12} \|\bw^e\|_{L^\infty(T)} \leq \sqrt{2} \|\bw^e\|_{L^\infty(T)}^{\frac12} h_{S_T}^{- \frac12}$ holds.
If  $\displaystyle \mathsf{Pe}_T > 1$, then $\delta_T^{- \frac12} \leq {\max(c^{\frac12}, \delta_0^{-\frac12} \|\bw^e\|_{L^\infty(T)}^{\frac12} h_{S_T}^{- \frac12})}$ holds.
 Using the assumption that the outer triangulation is quasi-uniform we get $ h_{S_T}^{-1} \lesssim h^{-1}$. Thus, we obtain
\[
\min\{\eps^{-\frac12}\|\bw^e\|_{L^\infty(T)},\delta_T^{-\frac12}\}\lesssim  { \max(c^{\frac12}, h^{-\frac12})\lesssim c^{\frac12}+ h^{-\frac12}},
\]
and
\begin{align}\label{aux40}
 I_4\lesssim&   \max_{x\in\Gamma_h}|\mathbf B_h-\bfP_h| \sum_{T\in\mathcal{F}_h}(\eps^{\frac12}\|\nath v_h\|_{L^2(T)} + \delta_T^{\frac12}\|\mathbf{w}^e_{\Gamma_h}\cdot\nath v_h\|_{L^2(T)})\notag\\
 &~~\times\min\{\eps^{-\frac12}\|\bw^e\|_{L^\infty(T)},\delta_T^{-\frac12}\}\|u^e\|_{L^2(T)}\notag\\
 \lesssim&  ~  {h(c^{\frac12}h+ h^{\frac12}) } \|v_h\|_{\ast}\|u\|_{L^2(\Gamma)}.
\end{align}
 Now we estimate $\varPi_1$.  Using the equation $-\eps \Delta_\Gamma u + \bw \cdot \nabla_\Gamma u + c u=f$ on $\Gamma$ we get
\begin{align}
\varPi_1&=\sum_{T\in\Fh} \delta_T \int_T \big( f\circ \bfp +\eps\Delta_{\Gamma_h}u^e - \bw^e\cdot\nath u^e -c u^e \big)\bw^e\cdot\nath v_h\, \ds \nonumber\\
&=\sum_{T\in\Fh}\delta_T \int_T \big(-\eps(\Delta_{\Gamma} u) \circ \bfp +\eps\Delta_{\Gamma_h}u^e \big) \bw^e\cdot\nath v_h \, \ds \nonumber\\
&\quad +\sum_{T\in\Fh}\delta_T \int_T \big((\bw^e\cdot \nat u) \circ \bfp -\bw^e\cdot\nath u^e\big)\bw^e\cdot\nath v_h \, \ds\nonumber\\
& \quad +\sum_{T\in\Fh}\delta_T \int_T \big(cu \circ \bfp - c u^e\big)\bw^e\cdot\nath v_h \, \ds \nonumber\\
&=:\varPi_1^1+\varPi_1^2+\varPi_1^3. \label{aux17}
\end{align}
 From $u^e = u\circ\bfp $ it follows  that $\varPi_1^3=0 $ holds. For $\varPi_1^2$ we obtain, using \eqref{hhl} and $|\mu_h^{-1} \mathbf B_h-\bfP_h| \leq |\mu_h -1| |\mu_h|^{-1} |\mathbf B_h| + |\mathbf B_h - \mathbf P_h| \lesssim h^2$,
\begin{align}
\varPi_1^2&=\sum_{T\in\Fh} \delta_T \int_T \big((\mu_h^{-1} \mathbf B_h-\bfP_h)\bw^e\cdot\nath u^e\big)\bw^e\cdot\nath v_h \, \ds \nonumber\\
& \lesssim h^2  \left(\sum_{T\in\Fh}\delta_T\|\bw^e\|_{L^\infty}^2 \| \nath u^e\|_{L^2(T)}^2\right)^{\frac{1}{2}}
\left(\sum_{T\in\Fh}\delta_T\|\bw^e \cdot \nath v_h\|_{L^2(T)}^2\right)^{\frac{1}{2}}\nonumber\\
&\lesssim h^{\frac{5}{2}}\|u\|_{H^2(\Gamma)}\|v_h\|_{\ast}. \label{aaa}
\end{align}
Since $\delta_T\eps\lesssim h^2$, we get
\begin{align}\label{aux19}\notag
 \varPi_2^1
 & \lesssim \left(\sum_{T\in\Fh}\eps^2\delta_T\| (\Delta_{\Gamma} u) \circ \bfp-\Delta_{\Gamma_h}u^e\|_{L^2(T)}^2\right)^{\frac{1}{2}}
\left(\sum_{T\in\Fh}\delta_T\|\bw^e\cdot\nath v_h\|_{L^2(T)}^2\right)^{\frac{1}{2}}\\
&\lesssim h\eps^{\frac12}\left(\sum_{T\in\Fh}\| (\Delta_{\Gamma} u) \circ \bfp-\Delta_{\Gamma_h}u^e\|_{L^2(T)}^2\right)^\frac{1}{2}\|v_h\|_{\ast}.
\end{align}
We finally consider the term between brackets in  \eqref{aux19}.
Using  the identity $\Div_{\Gamma} \mathbf f = \mathrm{tr} (\nat \mathbf f)$ and $\bn \cdot \nabla u^e( \bfp(\bfx))=0$ we obtain for $\bx \in T$, with $\nabla^2:= \nabla \nabla^T$.
\begin{equation}\label{Delta1}
\Delta_{\Gamma} u(\bfp(\bfx))=
\Div_{\Gamma} \nabla_\Gamma u(\bfp(\bfx))=\mathrm{tr}(\bfP\nabla \bfP\nabla u^e(\bfp(\bfx)))=
\mathrm{tr}(\bfP \nabla^2 u^e(\bfp(\bfx))\bfP).
\end{equation}
 From the same arguments it follows that
\begin{equation}\label{Delta2}
\Delta_{\Gamma_h} u^e(\bfx)=
\mathrm{tr}(\bfP_h\nabla \bfP_h \nabla u^e(\bfx))=
\mathrm{tr}(\bfP_h \nabla^2 u^e(\bfx)\bfP_h)
\end{equation}
holds.
Using \eqref{e:3.14n} and $|d(\bx)| \lesssim h^2$, $|\bfP- \bfP_h| \lesssim h$, $|\mathbf{H}| \lesssim 1$, $| \nabla\mathbf{H}| \lesssim 1$ we obtain
\begin{align*}
 \bfP_h \nabla^2 u^e(\bfx)\bfP_h  & =\bfP \nabla^2 u^e(\bfp(\bfx))\bfP +\mathbf{R}, \\
 |\mathbf{R}| & \lesssim  h \big(|\nabla^2 u^e(\bfp(\bfx)) | +|\nabla u^e(\bfp(\bfx)) | \big).
\end{align*}
Thus, using \eqref{Delta1} and \eqref{Delta2}, we get
\begin{align*}
 | \Delta_{\Gamma} u  (\bfp(\bx))-\Delta_{\Gamma_h}u^e(\bx)| & \leq \big| \mathrm{tr} (\bfP \nabla^2 u^e(\bfp(\bfx))\bfP-\bfP_h \nabla^2
 u^e(\bfx)\bfP_h) \big| \\
 & \lesssim  h \big(|\nabla^2 u^e(\bfp(\bfx)) | +|\nabla u^e(\bfp(\bfx)) | \big),
\end{align*}
and combining  this with  \eqref{aux19} yields
\[
\varPi_2^1\lesssim  h \eps^{\frac12}  \left(\sum_{T\in\Fh}\| (\Delta_{\Gamma} u) \circ \bfp -\Delta_{\Gamma_h}u^e\|_{L^2(T)}^2\right)^\frac{1}{2} \|v_h\|_\ast \lesssim\eps^{\frac12} h^2 \|u\|_{H^2(\Gamma)} \|v_h\|_\ast.
\]
Combining this with the results \eqref{aux12}-\eqref{aux40} and \eqref{aaa} proves the lemma.
\end{proof}

\subsection{Main theorem} \label{sectmain}

Now we put together  the results derived in the previous sections to prove the main result of the paper.
\begin{theorem}\label{Th1}
Let    Assumption~\ref{ass1} be satisfied. We consider problem parameters $\eps$ and $c$ as in \eqref{pardomain}. Assume that the solution $u$  of \eqref{e:2.5} has regularity $u\in H^2(\Gamma)$. Let $u_h$ be the discrete solution of the SUPG finite element method \eqref{e:2.8}.
Then the following holds:
\begin{equation}\label{errorEst}
\|u^e-u_h\|_{\ast}\lesssim h\big(h^{1/2}+\eps^{1/2}+c^{\frac12}h+\frac{h}{\sqrt{\eps+ c}}+ \frac{h^3}{\sqrt{\eps}}\big)  (\|u\|_{H^2(\Gamma)}+\|f\|_{L^2(\Gamma)}).
\end{equation}
\end{theorem}
\begin{proof} The triangle inequality yields
\begin{equation}
\|u^e-u_h\|_{\ast}\leq \|u^e-(I_h u^e)|_{\Gamma_h}\|_{\ast} + \|(I_h u^e)|_{\Gamma_h}- u_h \|_{\ast}.
\end{equation}
The second term in the upper bound can be estimated using Lemmas~\ref{thm0},~\ref{thm2},~\ref{thm3}:
\begin{align*}
  \frac12\|&\ipue- u_h \|^2_\ast \le a_h(\ipue- u_h,\ipue- u_h )\nonumber \\
&= a_h(\ipue -u^e, \ipue- u_h)+a_h(u^e-u_h, \ipue-u_h)\nonumber \\
&\lesssim h \big(h^{1/2}+\eps^{1/2}+c^{\frac12}h +\frac{h^2}{\sqrt{\eps+c}}+ \frac{h^3}{\sqrt{\eps}}\big) \|u\|_{H^2(\Gamma)}\|\ipue-u_h\|_{\ast}\nonumber \\
&\qquad + |a_{h}(u^e,\ipue-u_h)-f_{h}(\ipue-u_h)|\nonumber \\
&\lesssim h \big(h^{1/2}+\eps^{1/2}+c^{\frac12}h +\frac{h}{\sqrt{\eps+c}}+ \frac{h^3}{\sqrt{\eps}}\big)\big(\|u\|_{H^2(\Gamma)}+\|f\|_{L^2(\Gamma)}\big)\|\ipue-u_h\|_{\ast}.\nonumber
 \end{align*}
This results in
\begin{equation}\label{aux24}
  \|\ipue- u_h \|_{\ast}\lesssim  h \big(h^{1/2}+\eps^{1/2}+c^{\frac12}h +\frac{h}{\sqrt{\eps+c}}+ \frac{h^3}{\sqrt{\eps}}\big)(\|u\|_{H^2(\Gamma)}+\|f\|_{L^2(\Gamma)}).
\end{equation}
The error estimate  \eqref{errorEst} follows from \eqref{e:3.20} and \eqref{aux24}.\\
\end{proof}

\subsection{Further discussion}\label{sec_disc}
We comment on some aspects related to the main theorem. Concerning the analysis we note that the norm $\|\cdot\|_\ast$, which measures the error
on the left-hand side of  \eqref{errorEst}, is the standard SUPG norm as found in standard analyses of
planar streamline-diffusion finite element methods. The analysis in this paper contains new ingredients compared to the planar case. To control the geometric errors (approximation of $\Gamma$ by $\Gamma_h$) we derived a consistency error bound in Lemma~\ref{thm3}. To derive a continuity result (Lemma~\ref{thm2}), as in the planar case, we apply partial integration to the term $\int_{\Gamma_h} (\bw^e \cdot \nabla_{\Gamma_h} \phi) v_h \, \ds$, cf. \eqref{aux10}. However, different from the planar case, this results in jumps across the edges $E \in \mathcal{E}_h$ which have to be controlled, cf.~\eqref{aux22}. For this the new results in the Lemmas~\ref{lem:jumpvel} and \ref{lem_edge_est} are derived. These jump terms across the edges cause the term $\frac{h^4}{\sqrt{\eps}}$  in the error bound  in  \eqref{errorEst}.

Consider  the error reduction factor $h^{3/2}+\eps^{1/2}h+c^{\frac12}h^2+\frac{h^2}{\sqrt{\eps+ c}}+ \frac{h^4}{\sqrt{\eps}}$ on the right-hand side of \eqref{errorEst}. The first three terms of it are
typical for the error analysis of planar SUPG finite element methods for $P1$ elements. In the standard literature for the planar case, cf.~\cite{TobiskaBook}, one typically only considers the case $c >0$. Our analysis also applies to the case $c=0$, cf.~\eqref{pardomain}. Furthermore, the estimates are uniform w.r.t. the size of the parameter $c$.  For a fixed $c>0$ and $\eps \lesssim h$  the first four terms can be estimated by $\lesssim h^{3/2}$, a bound similar to the standard one for the planar case. The only ``suboptimal'' term is the last one, which is caused by (our analysis of) the jump terms. Note, however, that $\frac{h^4}{\sqrt{\eps}} \lesssim h^{3/2}$  if $h^5 \lesssim \eps$ holds, which is a very mild condition.

The norm $\|\cdot\|_\ast$  provides a robust control of streamline derivatives of the solution. Cross-wind oscillations, however, are not completely suppressed. It is well known that nonlinear stabilization methods can be used  to get control over cross-wind derivatives as well. Extending such methods to surface PDEs is not within the scope of the present paper.

The error estimates in this paper are in terms of the maximum mesh size over tetrahedra in $\omega_h$, denoted by $h$.
In practice, the stabilization parameter $\delta_T$ is based on the \emph{local} Peclet number and the stabilization
is switched off or reduced in the regions, where the mesh is ``sufficiently fine''. To prove error estimates  accounting for local smoothness of the solution $u$ and the local mesh size (as available for planar SUPG FE method), one needs local interpolation
properties of finite element spaces, instead of \eqref{e:3.18}, \eqref{e:3.19}. Since our finite element space is based on traces of piecewise linear functions, such local estimates are not immediately available. The extension of our analysis to this \emph{non}-quasi-uniform case is left for future research.

Finally, we remark on the case of a varying reaction  term coefficient $c$. If the coefficient  $c$ in the third term of \eqref{e:2.3} varies, the above analysis is valid with minor modifications. We briefly explain these modifications. The stabilization parameter $\delta_T$ should
be based on  elementwise  values $c_T=\max_{\bx\in T} c(\bx)$.
For the well-possedness of \eqref{e:2.3}, it is sufficient to assume $c$ to be strictly positive  on a subset of $\Gamma$ with positive measure:
\[
\mathcal{A}:=\operatorname{meas}\{\bx\in\Gamma\,:\, c(\bx)\ge {c}_0\}>0,
\]
with some $c_0>0$.
If this is satisfied,  the Friedrich's type inequality~\cite{Sobolev} (see, also Lemma~3.1 in \cite{Reusken08})
\begin{equation*}
\|v\|_{L^2(\Gamma)}^2 \le C_F(\|\nabla_\Gamma v\|_{L^2(\Gamma)}^2+\|\sqrt{c} v\|_{L^2(\Gamma)}^2)\quad
\text{for all}~v\in V
\end{equation*}
holds, with a constant $C_F$ depending on $c_0$ and $\mathcal{A}$. Using this, all arguments in the analysis can be generalized to the case of a varying $c(\bx)$ with obvious modifications.
 With $c_{\min}:=\text{ess\,inf}_{\bx\in \Gamma} c(\bx)$
and $c_{\max}:=\text{ess\,sup}_{\bx\in \Gamma} c(\bx)$, the final error estimate takes the form
 \begin{equation*}
\|u^e-u_h\|_{\ast}\lesssim h\big(h^{1/2}+\eps^{1/2}+c^{\frac12}_{\max}h+\frac{h}{\sqrt{\eps+ c_{\min}}}+ \frac{h^3}{\sqrt{\eps}}\big)  (\|u\|_{H^2(\Gamma)}+\|f\|_{L^2(\Gamma)}).
\end{equation*}

\section{Numerical experiments} \label{sectexperiments}
In this section we show results of a few  numerical experiments which illustrate the performance of the method.

\begin{example}
 \label{example1} \rm  The stationary problem \eqref{e:2.3} is solved on the unit sphere $\Gamma$, with
 the velocity field
$$
\mathbf{w}(\bx)=(-x_2\sqrt{1-x_3^2},x_1\sqrt{1-x_3^2},0)^T,
$$
which  is tangential to the sphere.
We set $\varepsilon=10^{-6}$, $c \equiv 1$ and consider  the solution
$$
u(\mathbf{x})= \frac{x_1 x_2}{\pi}\mathrm{arctan}\left(\frac{x_3}{\sqrt{\varepsilon}}\right).
$$
Note that $u$ has a sharp internal layer along the
equator of the sphere. The corresponding right-hand side function $f$ is given
by
$$
f(\bx)=\frac{8\varepsilon^{3/2}(2+\varepsilon+2 x_3^2)x_1x_2x_3}{\pi(\varepsilon+4 x_3^2)^2}
 + \frac{6\varepsilon x_1x_2+\sqrt{x_1^2+x_2^2}(x_1^2-x_2^2)}{\pi}\mathrm{arctan}\left(\frac{x_3}{\sqrt{\varepsilon}}\right)+u.
$$

We consider the standard (unstabilized)  finite element method in \eqref{plainFEM}  and the stabilized method \eqref{e:2.8}.
A sequence of meshes was obtained by the gradual refinement of the outer triangulation. The induced    \emph{surface} finite element spaces have dimensions  $N=448,\, 1864,\, 7552,\, 30412$.
The resulting algebraic systems are solved by a direct sparse solver. Finite element  errors  are  computed outside the layer: The variation of the  quantities $err_{L^2}=\|u-u_h\|_{L^2(D)}$,
$err_{H^1}=\|u-u_h\|_{H^1(D)}$, and $err_{L^{inf}}=\|u-u_h\|_{L^{\infty}(D)}$,
 with $D=\{\mathbf{x}\in\Gamma\ :\ |x_3|>0.3\}$, are shown in Figure~\ref{fig:err_ex1}. The results clearly show that the stabilized method performs much better than the standard one.
The results for the stabilized method indicate  a $\mathcal{O}(h^2)$ convergence in the $L^2(D)$-norm and $L^{\infty}(D)$-norm.   In the $H^1(D)$-norm
 a first order convergence   is observed.
 Note that the analysis predicts (only)  $\mathcal{O}(h^{\frac{3}{2}})$ convergence order in the (global) $L^2$-norm.

\begin{figure}[ht!]
\begin{center}
\includegraphics[width=\textwidth]{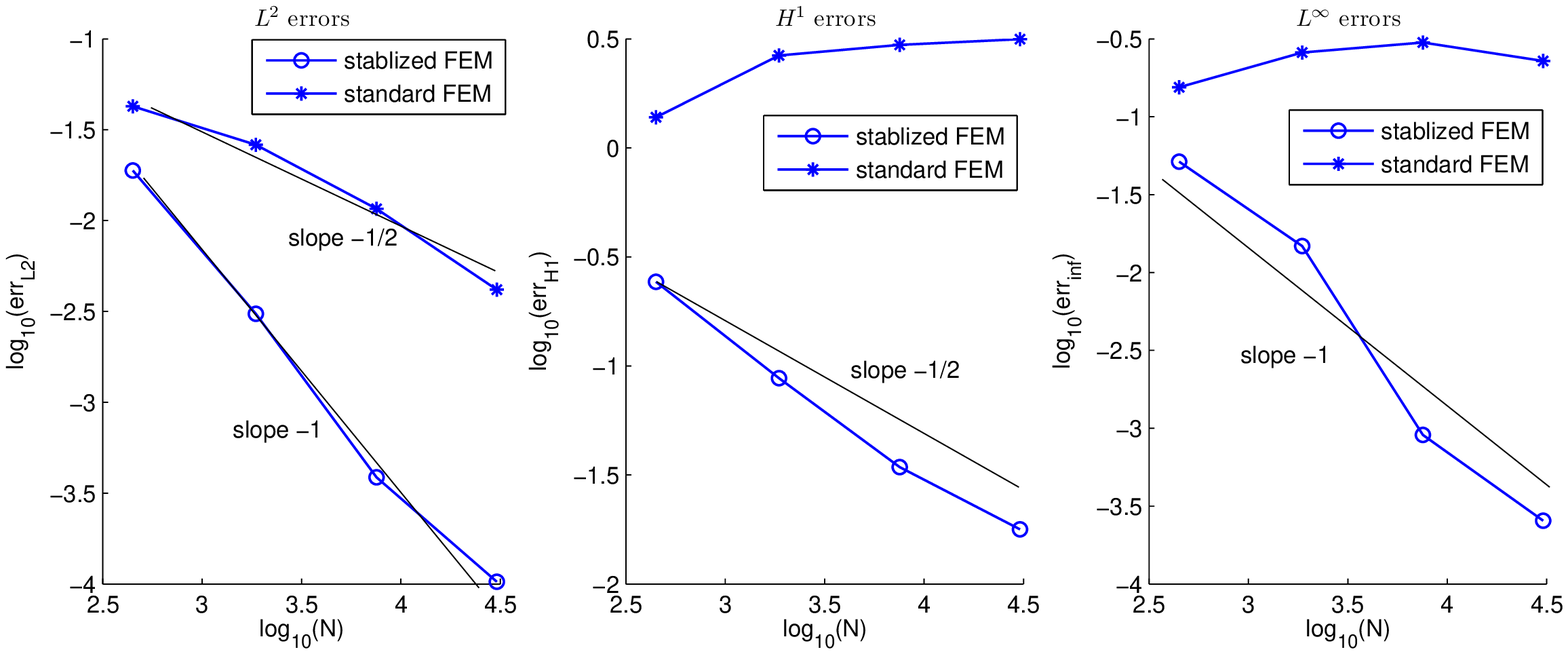}
\caption{Discretization errors for Example~\ref{example1}}\lbl{fig:err_ex1}
\end{center}
\end{figure}

Figure~\ref{fig:instable} shows the computed solutions with the two methods.
Since the layer is unresolved, the  finite element method \eqref{plainFEM} produces globally oscillating solution. The stabilized method gives a much better approximation, although the  layer
is slightly smeared, as is typical for the SUPG method.
\begin{figure}[ht!]
 \centering
  \subfigure[]{
    \includegraphics[width=2.2in]{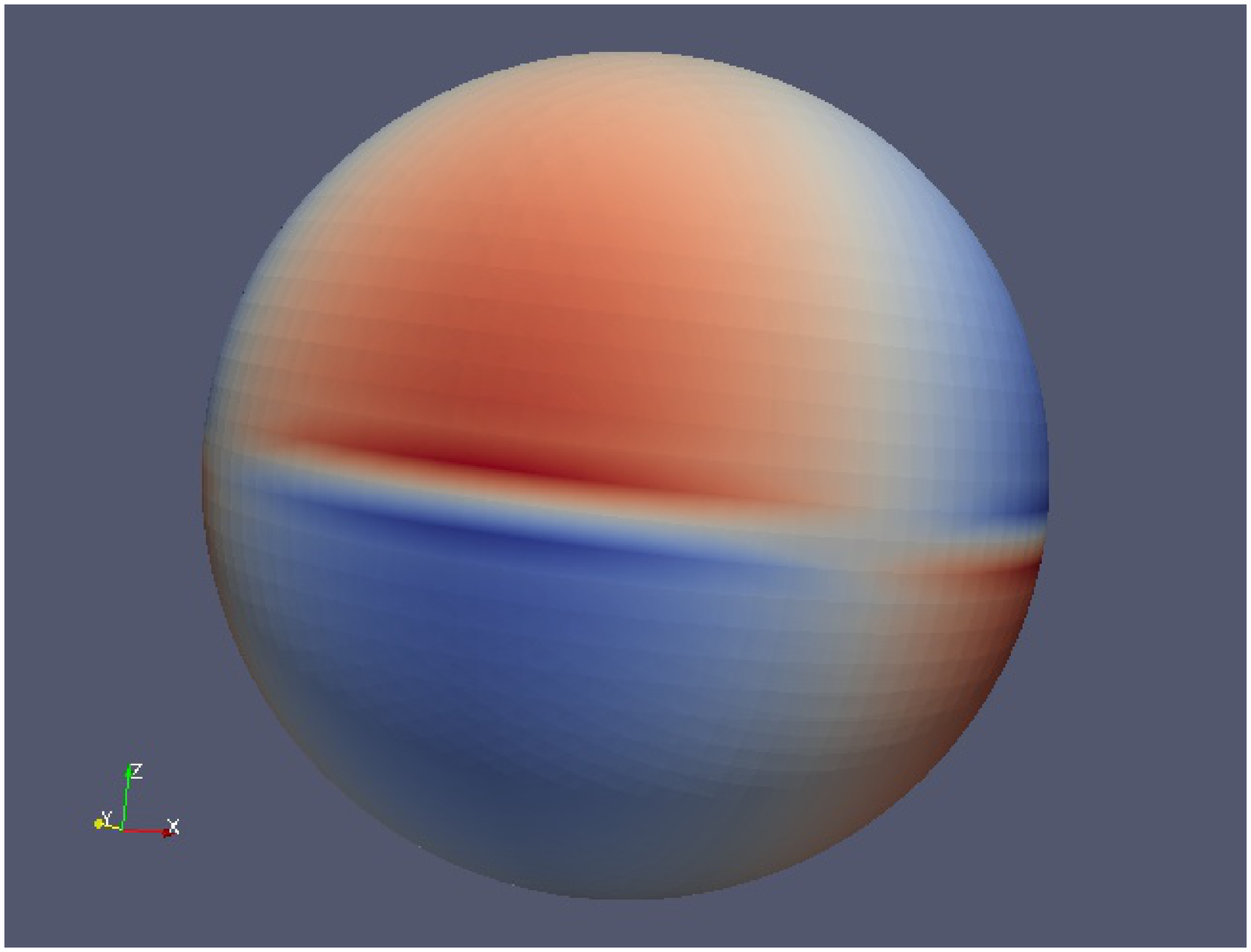}}
  \subfigure[]{
    \includegraphics[width=2.2in]{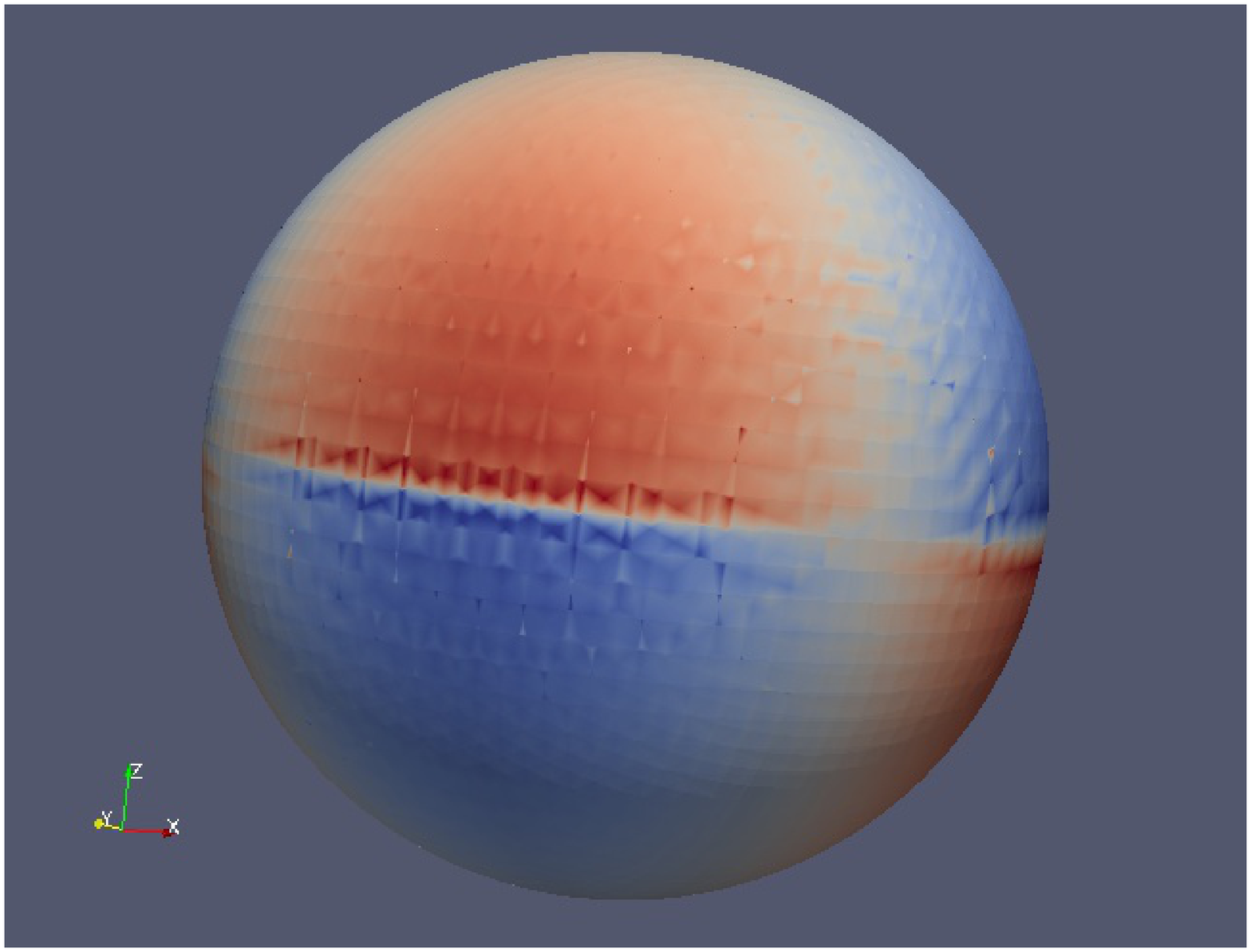}}
  \caption{Example~\ref{example1}: solutions using the stabilized method   and the  standard method \eqref{plainFEM}.}
  \lbl{fig:instable}
\end{figure}
\end{example}

\begin{example}
 \label{example2} \rm Now we consider the stationary problem \eqref{e:2.3} with $c \equiv 0$.
The problem is  posed on the unit sphere $\Gamma$,
with this same velocity field $\mathbf{w}$ as in Example~\ref{example1}.
We  set $\varepsilon=10^{-6}$,  and consider the solution
$$
u(\mathbf{x})= \frac{x_1x_2}{\pi}\mathrm{arctan}\left(\frac{x_3}{\sqrt{\varepsilon}}\right).
$$
The corresponding right-hand side function $f$ is now given by
$$
f(\bx)=\frac{8\varepsilon^{3/2}(2+\varepsilon+2 x_3^2)x_1x_2x_3}{\pi(\varepsilon+4 x_3^2)^2}
 + \frac{6\varepsilon x_1x_2+\sqrt{x_1^2+x_2^2}(x_1^2-x_2^2)}{\pi}\mathrm{arctan}\left(\frac{x_3}{\sqrt{\varepsilon}}\right).
$$

Note that for  $c = 0$ one looses explicit control of the $L_2$ norm in $\|\cdot\|_\ast$.
Thus we consider the streamline diffusion error :
$$
err_{SD}=\|\wt\cdot\nat (u-u_h)\|_{L^2(D)}.
$$
Results for this error quantity and for $err_{L^{inf}}=\|u-u_h\|_{L^{\infty}(D)}$   are shown in Figure~\ref{figg:err_ex2}.
We observe a  $\mathcal{O}(h)$ behavior for  the streamline diffusion error, which is consistent with
 our theoretical analysis. The  $L^{\infty}$-norm
of the error also shows a  first order of convergence.
\begin{figure}[ht!]
\begin{center}
\includegraphics[width=\textwidth]{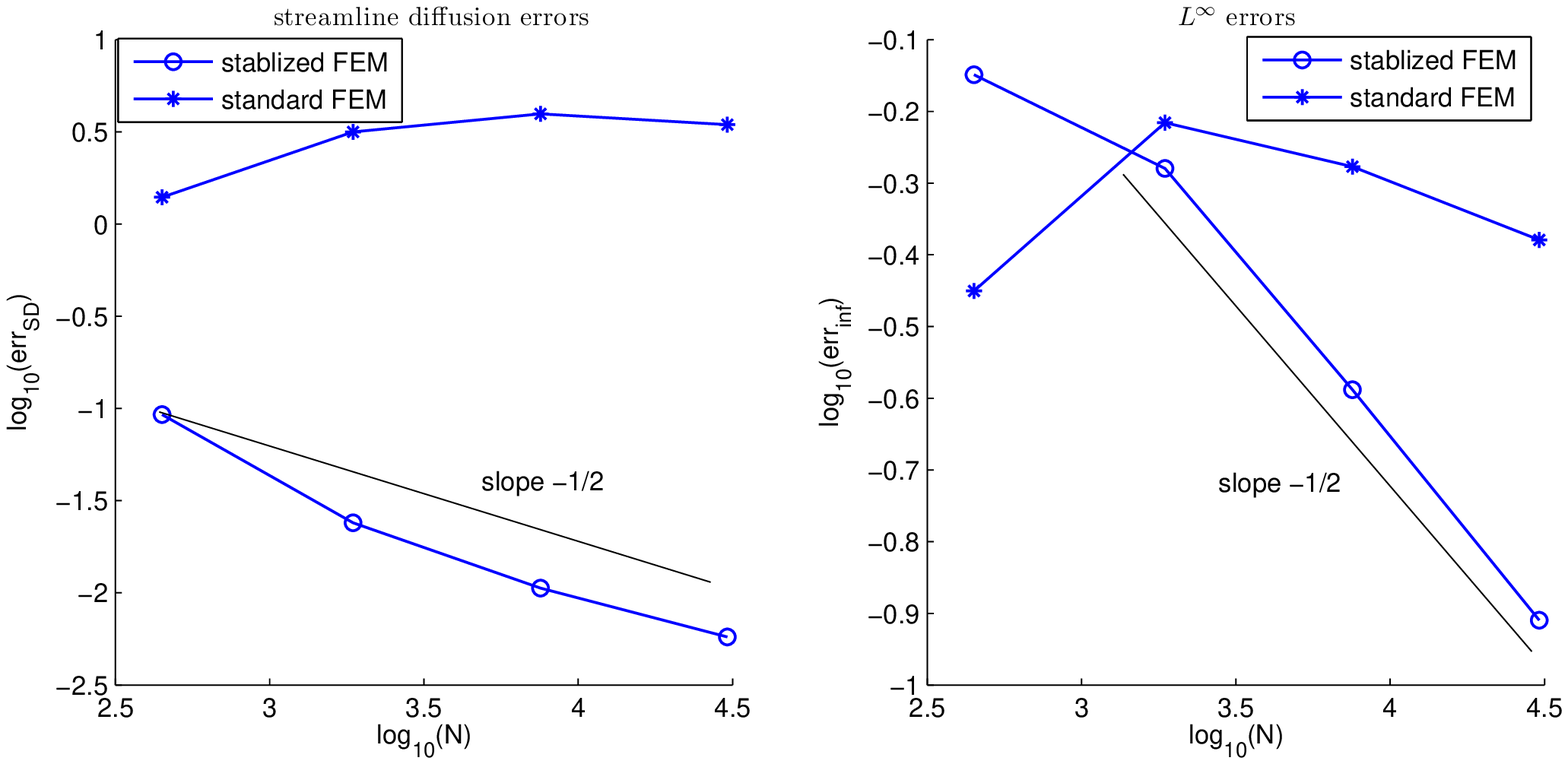}
\caption{Discretization errors for Example~\ref{example2}}\lbl{figg:err_ex2}
\end{center}
\end{figure}
\end{example}

\begin{example} \label{example3} \rm
We show how this stabilization can be applied to a time dependent problem and illustrate its stabilizing effect. We consider a non-stationary problem \eqref{e:2.2} posed on the  torus
\begin{equation}\label{e:torus}
\Gamma=\{(x_1,x_2,x_3)\ |\ (\sqrt{x_1^2+x_2^2}-1)^2+x_3^2=\frac{1}{16}\}.
\end{equation}
 We set $\varepsilon=10^{-6}$ and define  the advection field
$$
\mathbf{w}(\mathbf{x})=\frac{1}{\sqrt{x_1^2+x_2^2}}(-x_2,x_1, 0)^T,
$$
which is divergence free and satisfies $\bw \cdot \bn_\Gamma =0$.
The initial condition is
$$
u_0(\mathbf{x}) =  \frac{x_1x_2}{\pi}\mathrm{arctan}\left(\frac{x_3}{\sqrt{\varepsilon}}\right).
$$
The function $u_0$ possesses an internal layer, as shown  in Figure~\ref{fig:rotate}(a).

The stabilized spatial  semi-discretization of \eqref{e:2.2} reads: determine $u_h=u_h(t) \in V_h^\Gamma$ such that
\begin{equation} \label{spatttie}
m(\partial_t u_h, v_h) + \hat a_h(u_h,v_h)=0 \quad \text{for all}~~v_h \in V_h^\Gamma.
\end{equation}
with
\begin{equation*}
\begin{aligned}
m(\partial_t u,v):=&\int_{\Gamma_h}\partial_t u v \, \mathrm{d}\mathbf{s}+\sum_{T\in\mathcal{F}_h}\delta_{T}\int_{T} \partial_t u (\mathbf{w}^e\cdot\nabla_{\Gamma_h}v)\, \mathrm{d}\mathbf{s},\\
 \hat a_h(u,v):=& \varepsilon\int_{\Gamma_h}\nabla_{\Gamma_h} u\cdot\nabla_{\Gamma_h} v \, \mathrm{d}\mathbf{s} + \frac12\left[\int_{\Gamma_h}(\mathbf{w}^e\cdot\nabla_{\Gamma_h} u) v \, \mathrm{d}\mathbf{s} -\int_{\Gamma_h}(\mathbf{w}^e\cdot\nabla_{\Gamma_h} v) u \, \mathrm{d}\mathbf{s} \right]\\
 &+\sum_{T\in\mathcal{F}_h}\delta_T\int_{T}(-\varepsilon\Delta_{\Gamma_h}u + \mathbf{w}^e\cdot\nabla_{\Gamma_h} u)\mathbf{w}^e\cdot\nabla_{\Gamma_h} v \, \mathrm{d}\mathbf{s}.
\end{aligned}
\end{equation*}
Note that $ \hat a_h(\cdot,\cdot)$ is the same as $a_h(\cdot,\cdot) $ in \eqref{eqah} with $c \equiv 0$.
The resulting system of ordinary  differential equations  is discretized in time  by the  Crank-Nicolson scheme.

For $\varepsilon=0$ the exact solution is the transport  of $u_0(\mathbf{x})$
by a rotation around the $x_3$ axis. Thus the inner layer remains the same for all $t>0$. For  $\varepsilon=10^{-6}$, the exact solution is
similar, unless $t$ is large enough for dissipation to play a noticeable  role.
The space $V_h^\Gamma$ is constructed in the same way as in the previous examples. The spatial discretization has 5638 degrees of freedom.
The fully discrete problem is obtained by combining the SUPG method in \eqref{spatttie} and the Crank-Nicolson method with time step $\delta t=0.1$.
The evolution of the solution is illustrated in Figure~\ref{fig:rotate} demonstrating a smoothly `rotated' pattern.
\begin{figure}[ht!]
 \centering
  \subfigure[]{
    \includegraphics[width=2.1in]{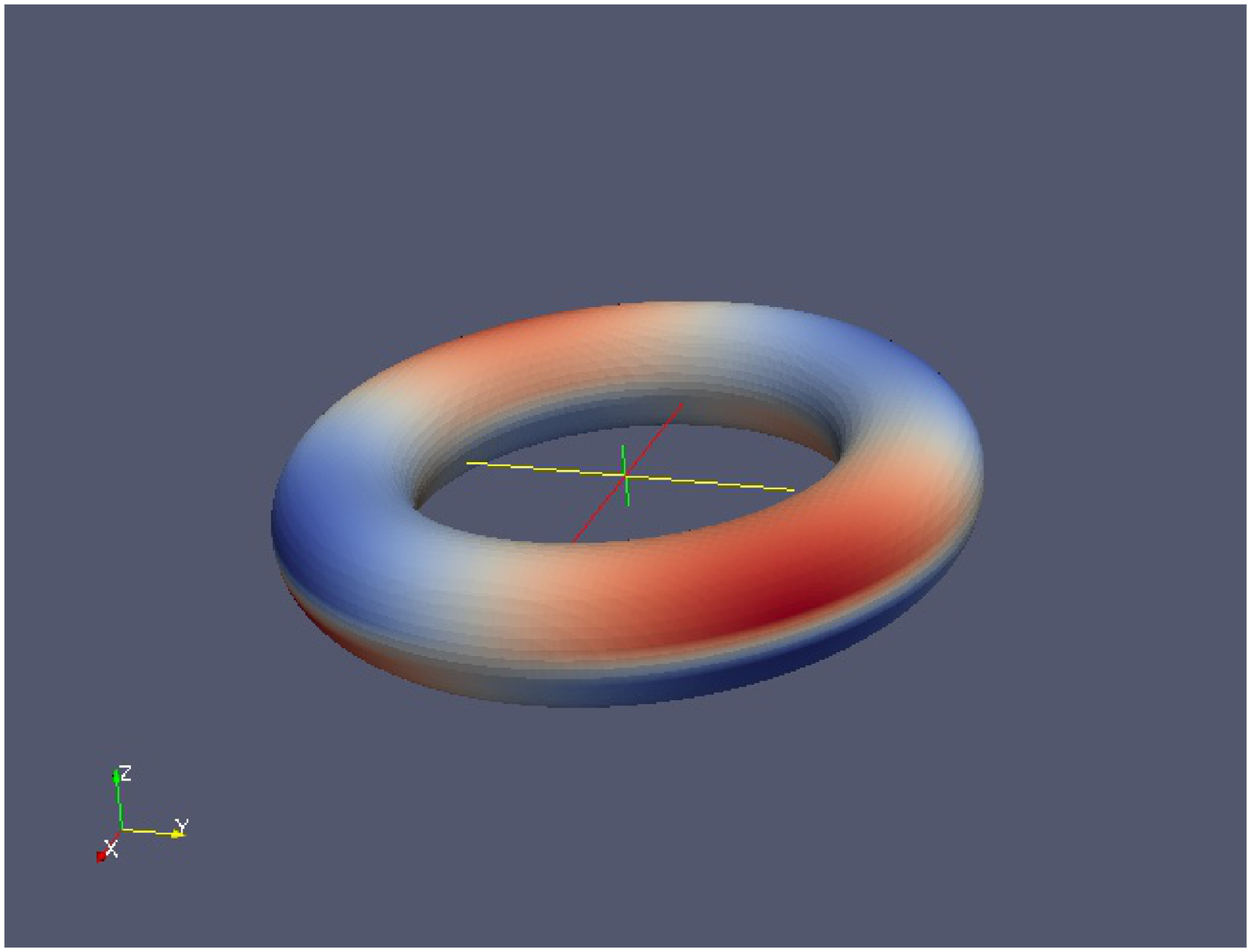}}
  \subfigure[]{
    \includegraphics[width=2.1in]{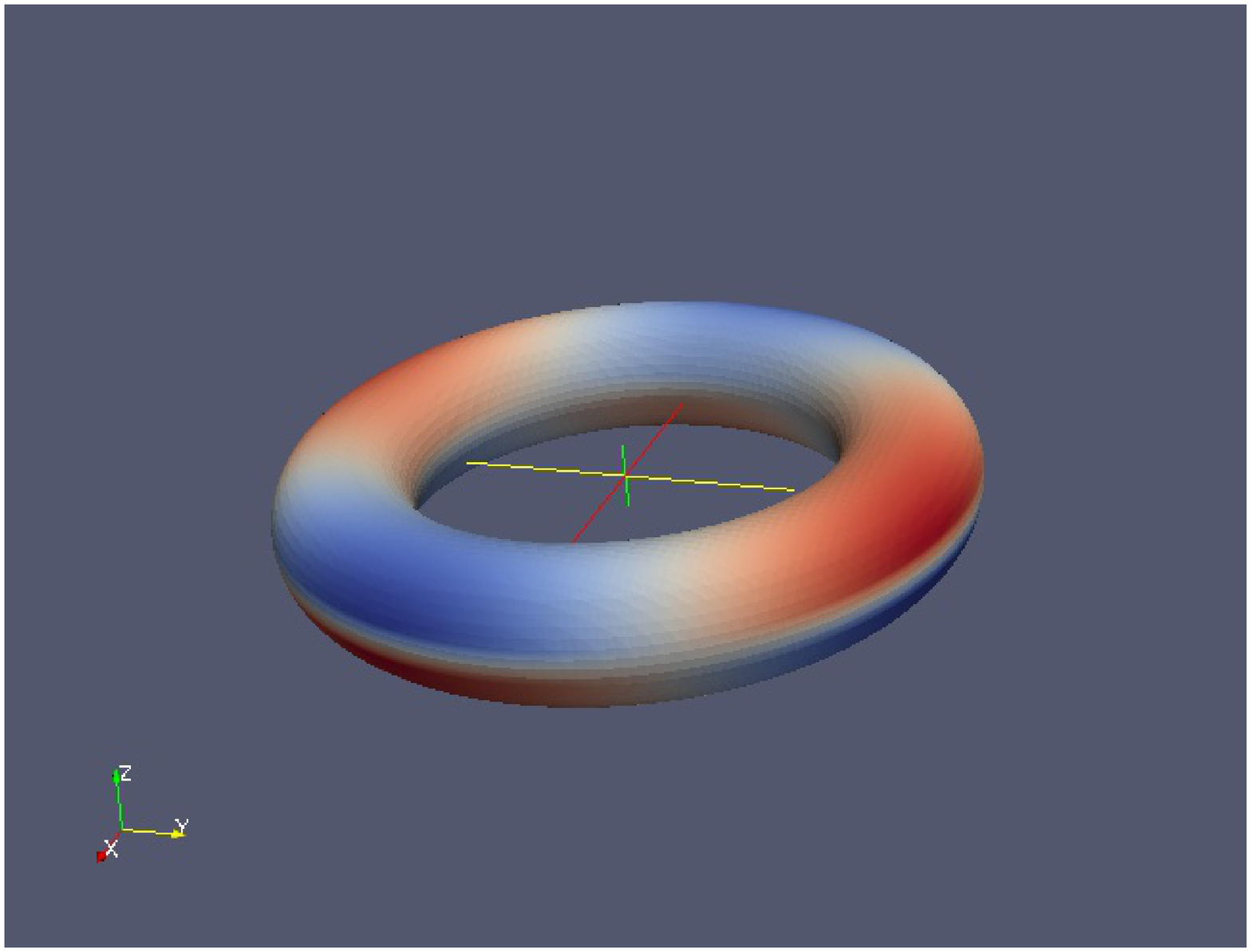}}
  \subfigure[]{
    \includegraphics[width=2.1in]{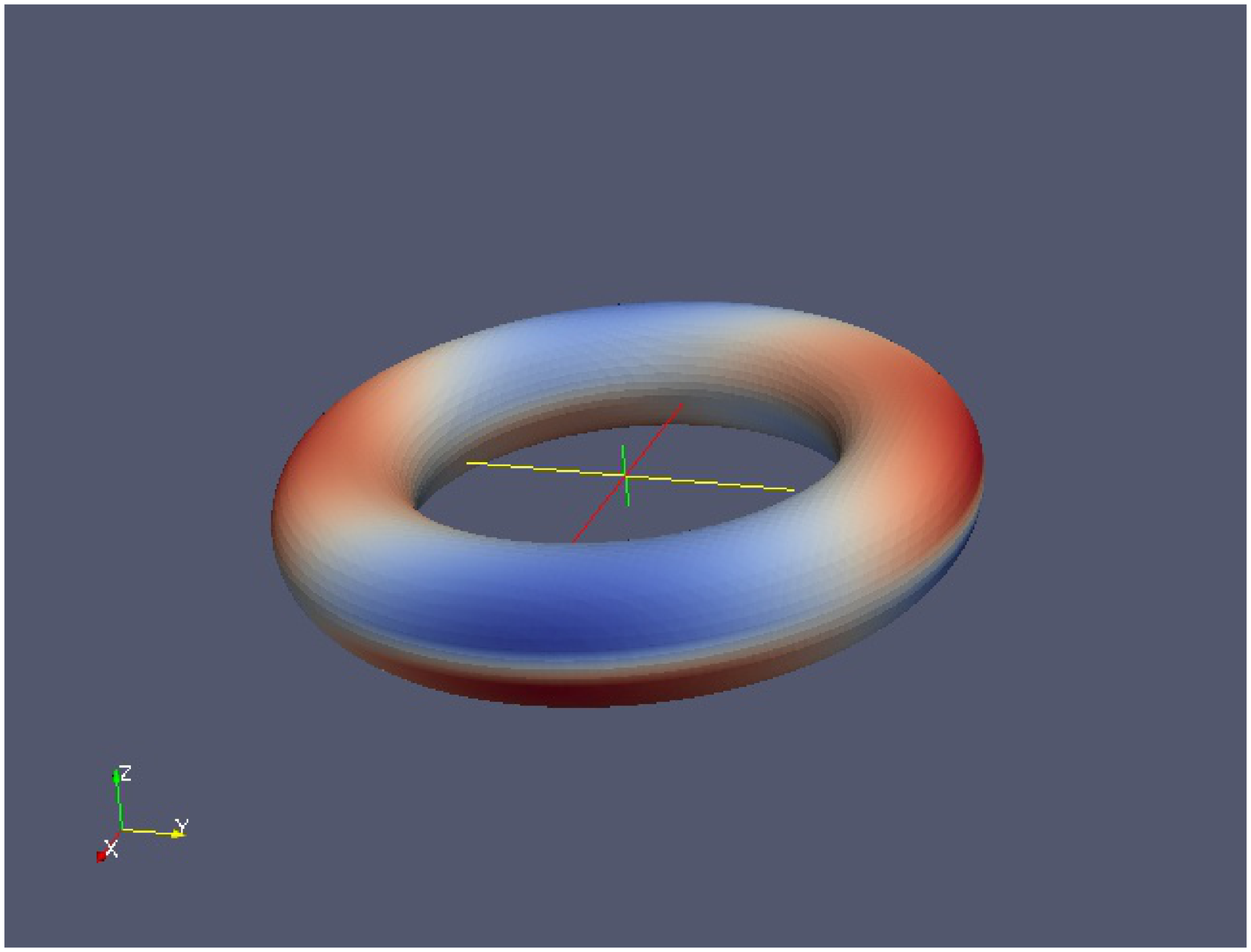}}
  \subfigure[]{
    \includegraphics[width=2.1in]{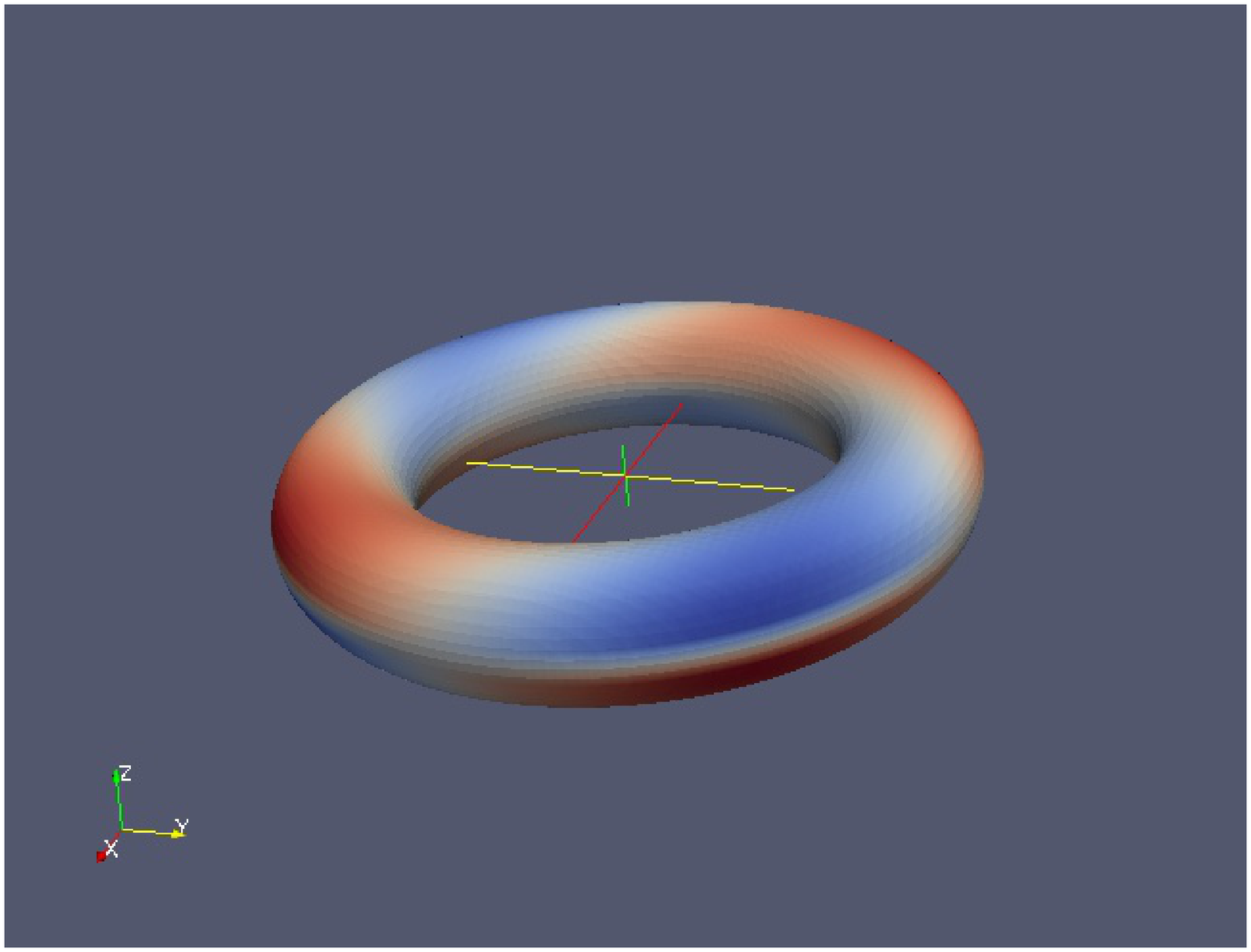}}
 \caption{Example~\ref{example3}: solutions for $t=0,0.6,1.2,1.8$ using the SUPG stabilized FEM.}\lbl{fig:rotate}
\end{figure}

We repeated this experiment with $\delta_T=0$ in  the bilinear forms $m(\cdot,\cdot)$ and $\hat a_h(\cdot, \cdot)$ in \eqref{spatttie}, i.e. the method without stabilization.
 As expected, we obtain  (on the same grid) much less smooth discrete solutions (Figure~\ref{fig:rotate1}).

\begin{figure}[ht!]
 \centering
  \subfigure[]{
    \includegraphics[width=2.1in]{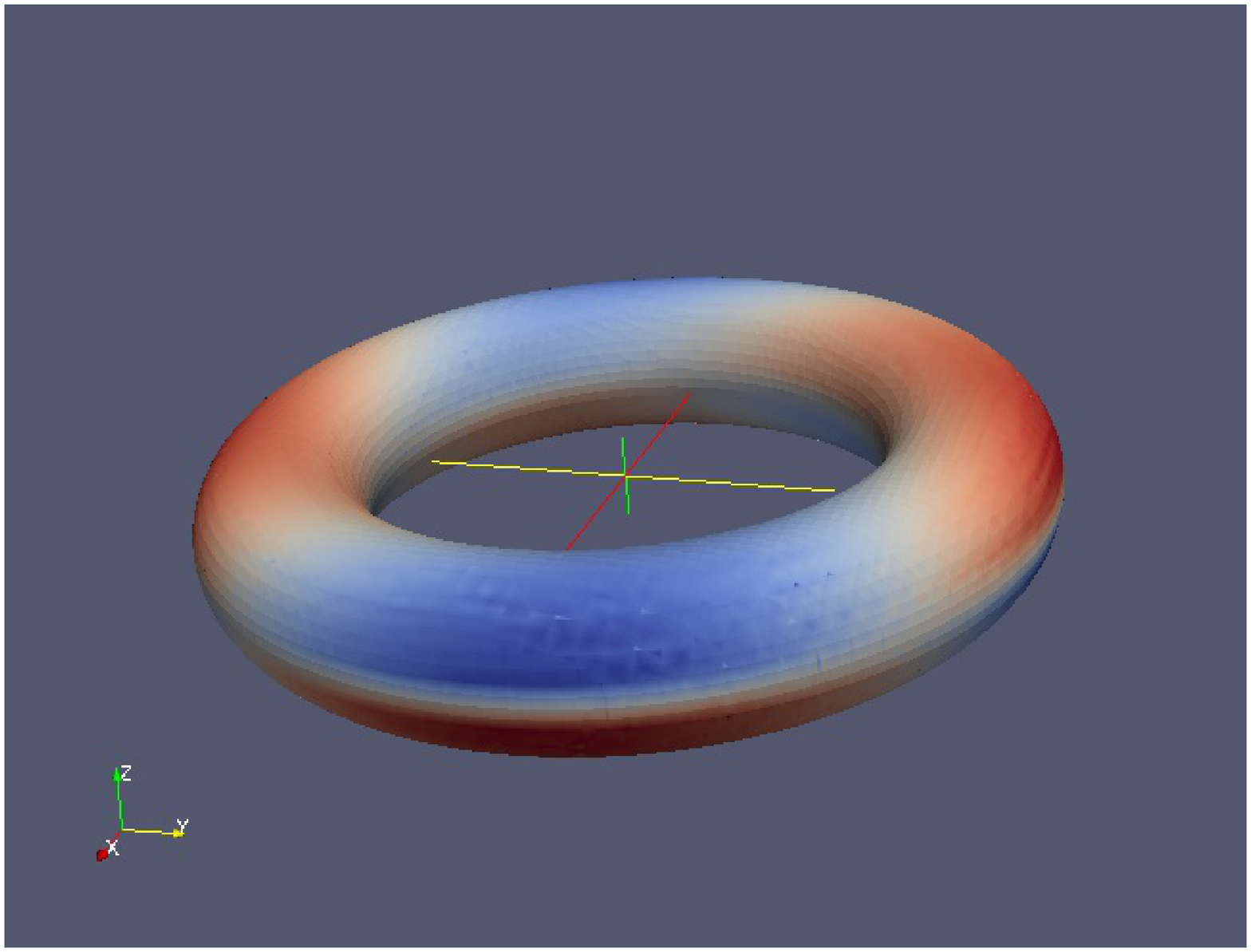}}
  \subfigure[]{
    \includegraphics[width=2.1in]{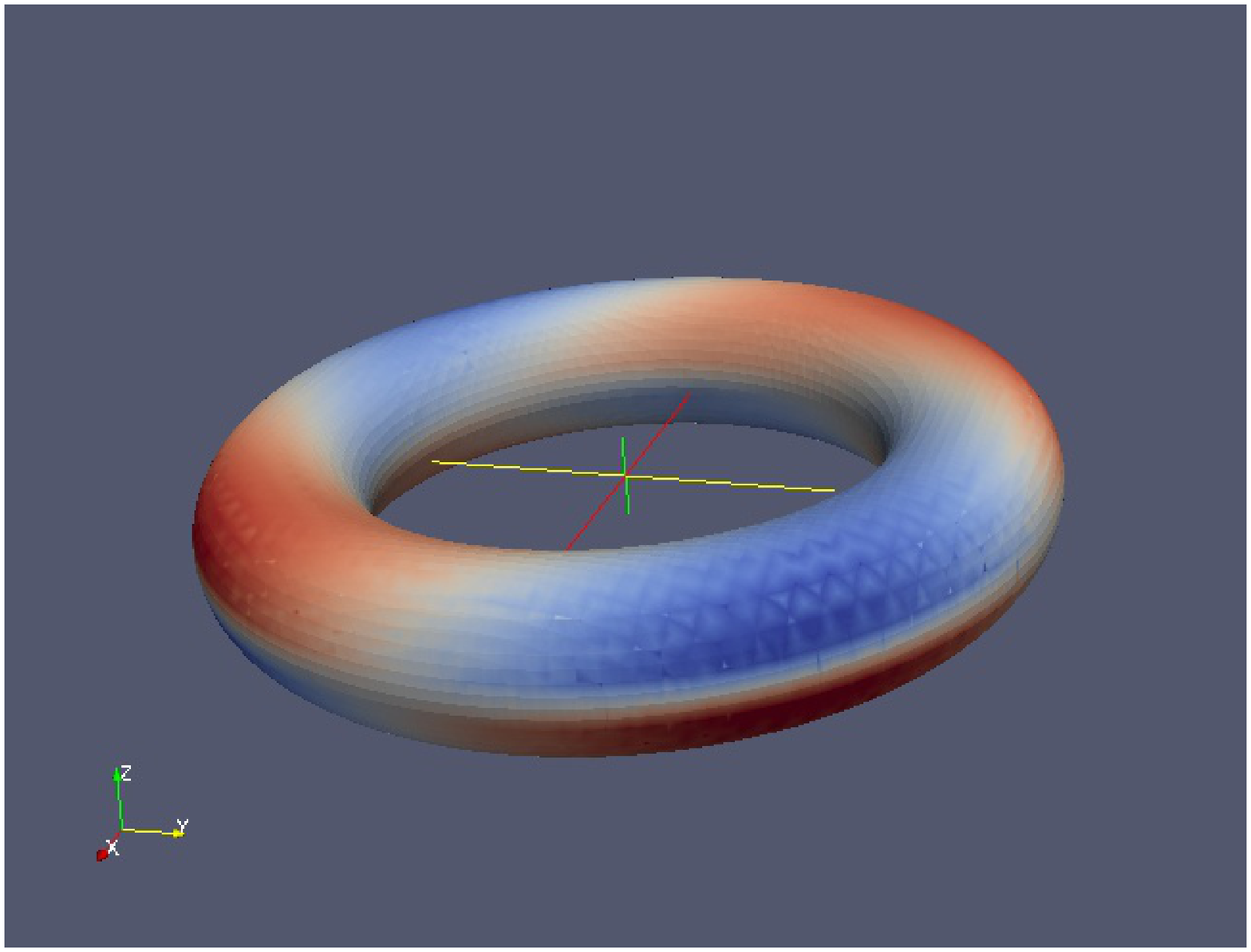}}
 \caption{Example~\ref{example3}: solutions for $t=1.2,1.8$ using the standard FEM.}\lbl{fig:rotate1}
\end{figure}

\begin{remark} \rm
With respect to mass conservation  of the scheme we note the following.  For  $v_h\equiv 1$ in \eqref{spatttie} we  get, with $M_h(t):=\int_{\Gamma_h}  u_h(x,t) \, \ds$, 
\begin{equation*}
 \begin{aligned}
|\frac{d}{dt} M_h(t)|&= |\int_{\Gamma_h} \frac12\mathbf{w}^e\cdot\nabla_{\Gamma_h} u_h\, \ds |\\
& =|-\frac12\sum_{E\in\mathcal{E}_h}\int_{E}\mathbf{w}^e\cdot[\mathbf{m}] u_h\, \ds+\frac12\int_{\Gamma_h} \divth\bw^e u_h\, \ds|
\\
&\lesssim h^2\sum_{E\in\mathcal{E}_h}\int_{E}|u_h|\, \ds+Ch\int_{\Gamma_h}|u_h|\, \ds.
 \end{aligned}
\end{equation*}
Here we used estimates from Lemmas 3.1 and  3.5. Using  Lemma 3.6 we get 
\begin{align*} \sum_{E\in\mathcal{E}_h}\int_{E}|u_h| \, \ds & \lesssim
h^{- \frac12} \Big( \sum_{E\in\mathcal{E}_h}\int_{E} u_h^2\ds \Big)^\frac12  \\ & \lesssim h^{-1}\big( \|u_h\|_{L^2(\Gamma_h)} + h \|\nabla_{\Gamma_h} u_h \|_{L^2(\Gamma_h)}).
\end{align*}
Assume that for the discrete solution we have a bound $\|u_h\|_{L^2(\Gamma_h)}   + h \|\nabla_{\Gamma_h} u_h \|_{L^2(\Gamma_h)} \leq c$ with $c$ independent of $h$. Then we obtain $|\frac{d}{dt} M(t)| \lesssim h$ and thus
$|M_h(t)-M_h(0)| \leq c  t h$, with a constant $c$ independent of $h$ and $t$. 
  Hence, with respect to mass conservation we 
  have an error that is (only)  first order in $h$. Concerning mass conservation it would be better to use a discretization in which
  in the discrete bilinear form in \eqref{eqah} one replaces
 \begin{equation} \label{convform}
 \frac12\left[\int_{\Gamma_h}(\bw^e\cdot\nath u) v \, \ds -\int_{\Gamma_h}(\bw^e\cdot\nath v) u\,  \ds \right]\quad\text{by}\quad -\int_{\Gamma_h}(\bw^e\cdot\nath v) u\,  \ds.
\end{equation}
This method results in optimal mass conservation: $\frac{d}{dt} M_h(t)=0$.
It turns out, however, that (with our approach) it is more difficult to analyze. In particular, it is not clear how to derive a satisfactory coercivity bound. 

In numerical experiments we observed  that the behavior of the two methods (i.e, with the two variants given in \eqref{convform}) is very similar. In particular, the mass conservation error bound  $|M(t)-M(0)| \leq c  t h$ for the first method seems to be too pessimistic (in many cases). To illustrate this, we show results for the problem described above, but  with  initial condition 
$$
u_0(\mathbf{x}) = 1+ \frac{1}{\pi}\mathrm{arctan}\left(\frac{x_3}{\sqrt{\varepsilon}}\right),\quad
\int_{\Gamma} u_0 \, \ds=\pi^2\approx 9.8696.
$$
Figure~\ref{fig:mass} shows the quantity $M_h(t)$ for several mesh sizes $h$. For $t=0$ we have, due to interpolation of the initial condition $u_0$,  a difference between $M_h(0)$ and $\int_{\Gamma} u_0 \ds$ that is of order $h^2$. For $t >0$ we see, except for the very coarse mesh with $h=1/4$ a very good mass conservation.
\end{remark}

\begin{figure}[ht!]
\begin{center}
\includegraphics[width=0.7\textwidth]{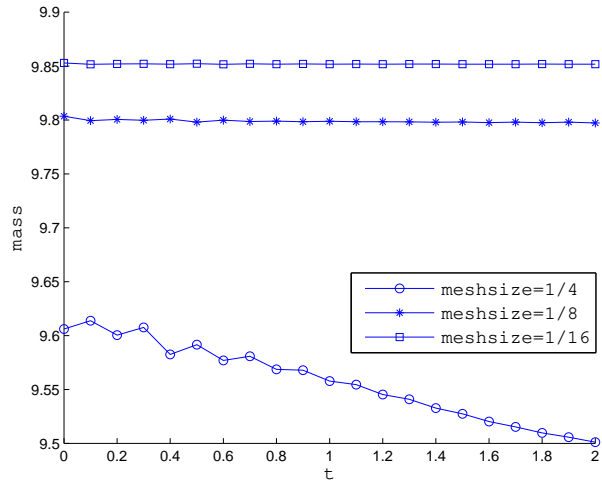}
\caption{Total mass variation for Example~\ref{example3}}\lbl{fig:mass}
\end{center}
\end{figure}
\end{example}

\begin{example}
  \label{example4} \rm As a final illustration we show results for the non-stationary problem \eqref{e:2.2}, but now on a surface  with a ``less regular''  shape. We take the surface given in \cite{Dziuk88}:
\begin{equation}\label{e:Dziuk}
\Gamma=\{(x_1,x_2,x_3)\ |\ (x_1-x_3^2)^2+x_2^2+x_3^2=1\}.
\end{equation}
 We set $\varepsilon=10^{-6}$ and define the advection field as the $\Gamma$-tangential part of
$
\tilde{\mathbf{w}}=(-1,0,0)^T.
$
This velocity field  does not satisfy the divergence free condition.
The initial condition is taken as
$
u_0(\mathbf{x}) =  1.
$
We apply the same method as in Example~\ref{example3}. The mesh size is $1/8$ and the time step is $\delta t=0.1$.
Figure~\ref{fig:dziuk} shows the solution for several $t$ values. We observe that as time evolves mass is transported from the two poles on the right to  the left pole,
just as expected. Our discretization yields for this strongly convection-dominated transport problem a qualitatively good discrete result, even with a (very) low grid resolution. 
\begin{figure}[ht!]
 \centering
  \subfigure[]{
    \includegraphics[width=2.1in]{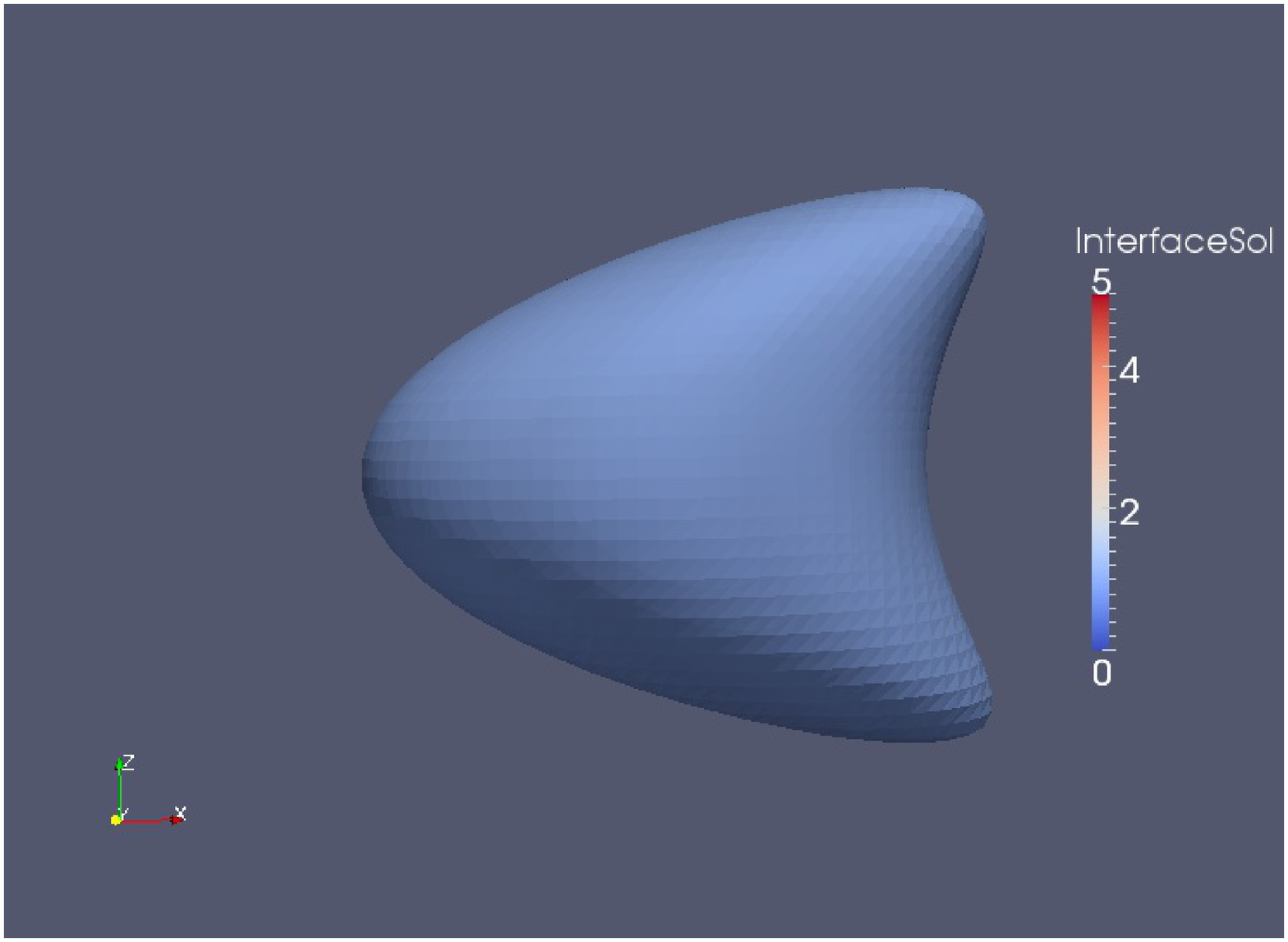}}
  \subfigure[]{
    \includegraphics[width=2.1in]{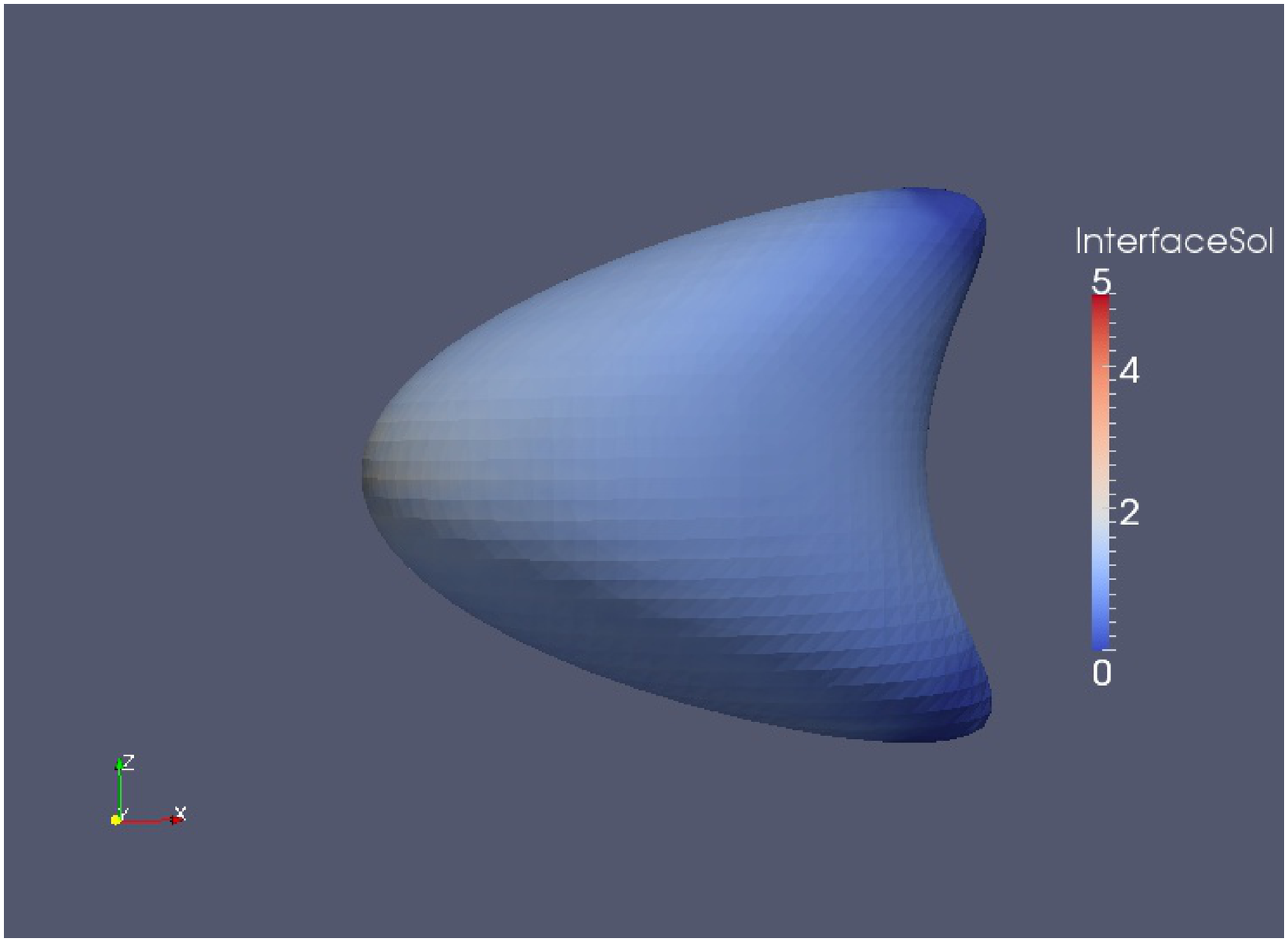}}
  \subfigure[]{
    \includegraphics[width=2.1in]{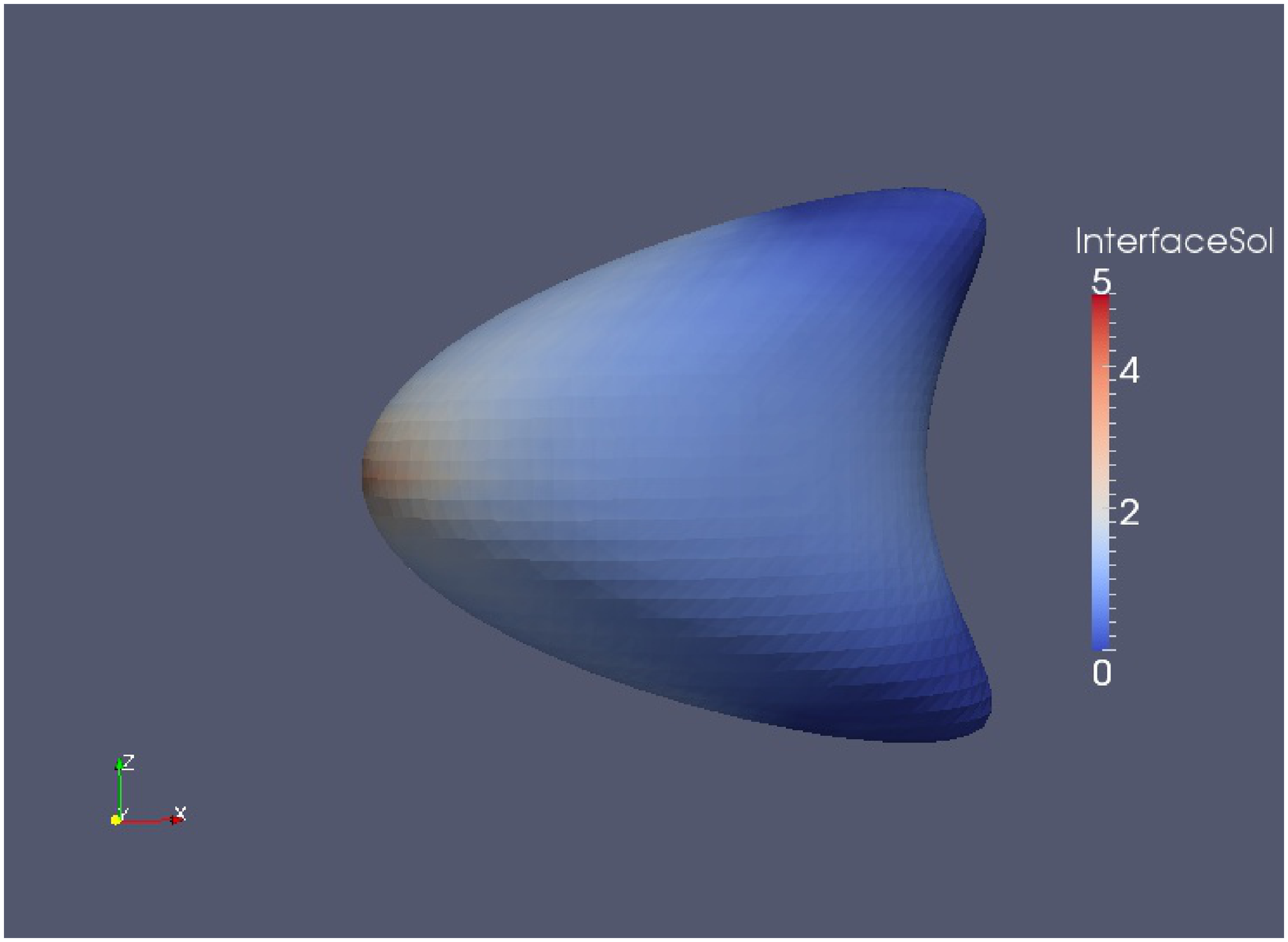}}
  \subfigure[]{
    \includegraphics[width=2.1in]{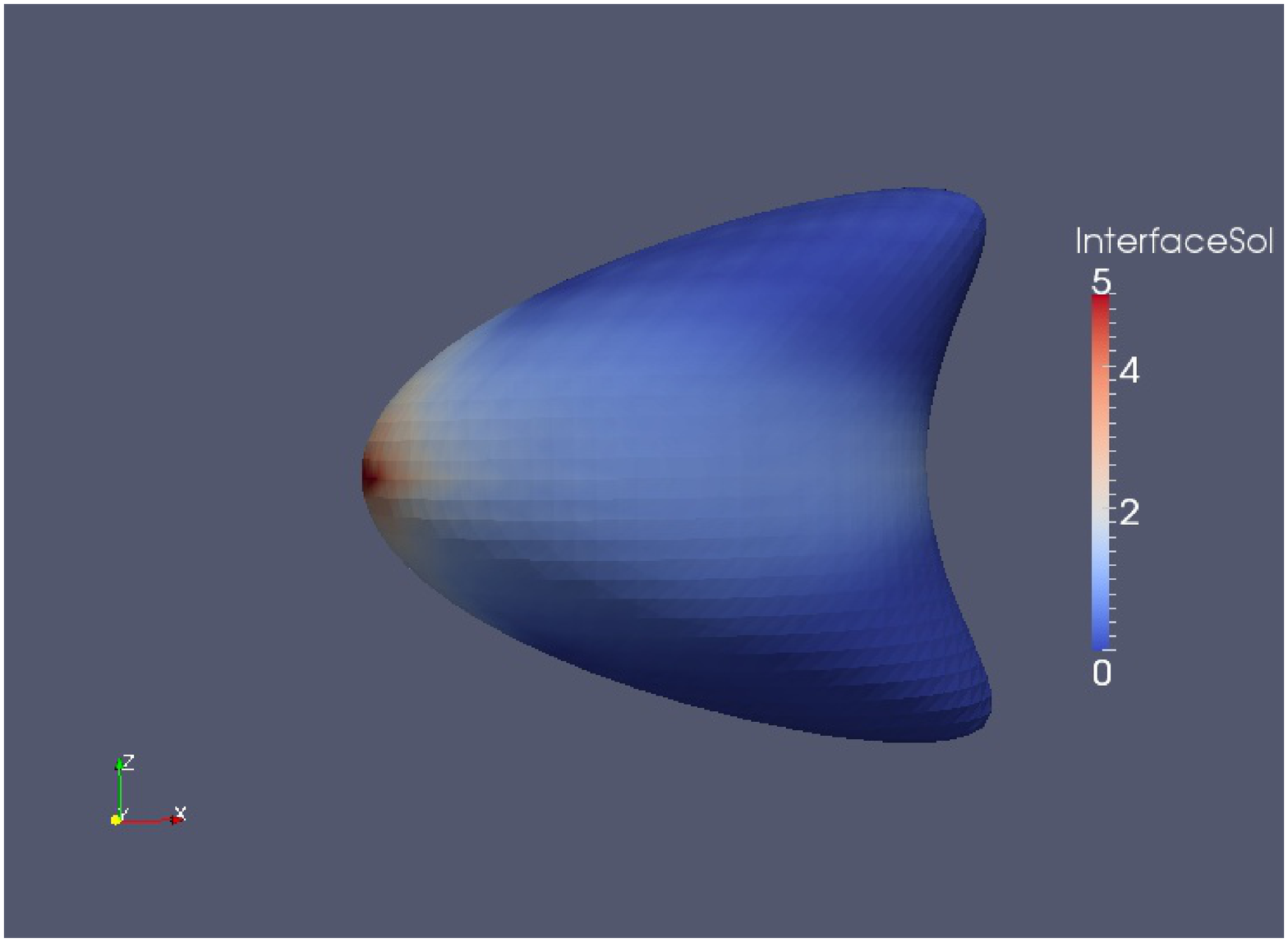}}
 \caption{Example~\ref{example4}: solutions for $t=0,0.5,1.0,2.0$ using the SUPG stabilized FEM.}\lbl{fig:dziuk}
\end{figure}
\end{example}

\section*{Acknowledgments}
This work has been supported in part by the DFG  through grant RE1461/4-1, by the Russian Foundation for  Basic Research through grants 12-01-91330, 12-01-00283, and by NSFC project 11001260.
We thank J. Grande for his support with the implementation of the methods  and C. Lehrenfeld for useful  discussions.

\bibliographystyle{unsrt}
\bibliography{literatur}
\end{document}